\documentclass[noinfoline,preprint]{article}
\usepackage{imsart}
\usepackage[utf8]{inputenc} 
\usepackage[OT1]{fontenc}    
\usepackage{hyperref}       
\usepackage{url}            
\usepackage{booktabs}       
\usepackage{amsfonts}       
\usepackage{amsmath,amsthm,amsfonts}
\usepackage{nicefrac}       
\usepackage{microtype}      

\usepackage{mathtools}
\usepackage[psamsfonts]{amssymb}
\usepackage{mathrsfs}
\usepackage{graphicx}
\usepackage{dsfont}
\usepackage{wrapfig}
\usepackage{comment}
\usepackage{rotating}
\usepackage{subcaption}
\usepackage[font={small}]{caption}  
\usepackage{algorithmic}
\usepackage[ruled,linesnumbered,shortend,vlined,noend]{algorithm2e}
\usepackage[all]{xy}
\usepackage{xcolor}
\usepackage[capitalize]{cleveref}
\usepackage{arydshln}
\usepackage{enumitem}
\setlist[itemize]{leftmargin=*}
\usepackage{wrapfig}  
\usepackage{listings}

\lstdefinestyle{cust}{
language=python,
commentstyle=\ttfamily,
basicstyle=,
escapeinside={\%*}{*)},
frame=single,
keepspaces=true,
keywordstyle=\bfseries,
morekeywords={*,Input,Output},
}
\newtheorem{theorem}{Theorem}

\newtheorem{lem}[theorem]{Lemma}
\newtheorem{corollary}[theorem]{Corollary}
\newtheorem{defi}[theorem]{Definition}

\newtheorem{prop}[theorem]{Proposition}
\newtheorem*{lem*}{Lemma}

\newcommand{\cb}{\mathbf{c}}
\newcommand{\R}{\mathbb{R}}
\newcommand{\E}{\mathbb{E}}

\renewcommand{\1}{\mathds 1}

\newcommand{\B}{\mathcal{B}}

\newcommand{\Var}{\operatorname{Var}}

\begin{document}

\begin{frontmatter}
\title{Optimal quantization of the mean measure and applications to statistical learning}
\author{Fr\'ed\'eric Chazal, Cl\'ement Levrard, and Martin Royer}
\date{}
\maketitle
\runtitle{Learning with quantized mean measure}

%
\begin{abstract}
This paper addresses the case where data come as point sets, or more generally as discrete measures. Our motivation is twofold: first we intend to approximate with a compactly supported measure the mean of the measure generating process, that coincides with the intensity measure in the point process framework, or with the expected persistence diagram in the framework of persistence-based topological data analysis. To this aim we provide two algorithms that we prove almost minimax optimal. 

Second we build from the estimator of the mean measure a vectorization map, that sends every measure into a finite-dimensional Euclidean space, and investigate its properties through a clustering-oriented lens. In a nutshell, we show that in a mixture of measure generating process, our technique yields a representation in $\mathbb{R}^k$, for $k \in \mathbb{N}^*$ that guarantees a good clustering of the data points with high probability. Interestingly, our results apply in the framework of persistence-based shape classification via the ATOL procedure described in \cite{Royer19}.  
     At last, we assess the effectiveness of our approach on simulated and real datasets, encompassing text classification and large-scale graph classification.
\end{abstract}
\end{frontmatter}

\section{Introduction}
This paper handles the case where we observe $n$ i.i.d measures $X_1, \hdots, X_n$ on $\R^d$, rather than $n$ i.i.d sample points, the latter case being the standard input of many machine learning algorithms. Such kind of observations naturally arise in many situations, for instance when data are spatial point patterns: in species distribution modeling \cite{Renner15}, repartition of clusters of diseases \cite{Diggle90}, modelisation of crime repartition \cite{Shirota17} to name a few. The framework of i.i.d sample measures also encompasses analysis of multi-channel time series, for instance in embankment dam anomaly detection from piezometers \cite{Jung15}, as well as topological data analysis,  where persistence diagrams naturally appear as discrete measures in $\R^2$ \cite{Chazal09,Chazal2016}. 

The objective of the paper is the following: we want to build from data an embedding of the sample measures into a finite-dimensional Euclidean space that preserves clusters structure, if any. Such an embedding, combined with standard clustering techniques should result in efficient measure clustering procedures. It is worth noting that clustering of measures has a wide range of possible applications: in the case where data come as a collection of finite point sets for instance, including ecology \cite{Renner15}, genetics \cite{Royer19,Adams17}, graphs clustering \cite{Carriere19,Tran18} and shapes clustering \cite{Chazal09}.

Our vectorization scheme is quite simple: for a $k$-points vector $\cb=(c_1, \hdots, c_k)$ in $(\R^d)^k$, we map each measure $X_i$ on $\R^d$ into a vector $v_i \in \R^k$ that roughly encodes how much weight $X_i$ spreads around every $c_j$. Then, in a mixture distribution framework, we provide general conditions on the $k$-points vectors that allow an efficient clustering based on the corresponding vectorization. Thus, efficiency of our measure clustering procedure essentially depends    on our ability to find discriminative $k$-points vectors from sample.

Choosing a fixed grid of $\R^d$ as vectorization points is possible, however the dimension of such a vectorization is constrained and can grow quite large. Another choice could be to draw vectorization points at random from the sample measures. Note however that such a method could miss discriminative areas. We choose an intermediate approach, by building a $k$-points approximation of a central measure of $X_1, \hdots, X_n$. Note that finding compact representations of central tendencies of measures is an interesting problem in itself, especially  when the sample measures are naturally organized around a central measure of interest, for instance in image analysis \cite{Cuturi13} or point processes modeling \cite{Renner15,Diggle90,Shirota17}. 

In \cite{Cuturi13}, the central measure is defined as the Wasserstein barycenter of the distribution of measures. Namely, if we assume that $X_1, \hdots, X_n$ are i.i.d measures on $\R^d$ drawn from $X$, where $X$ is a probability distribution on the space of measures, then the central measure is defined as $\mu_{W} = \arg\min_{\nu} \mathbb{E} \left ( W_2(X,\nu) \right )^2$, where $\nu$ ranges in the space of measures and $W_2$ denotes the Wasserstein distance. Note that this definition only makes sense in the case where $X(\R^d)$ is constant a.s., that is when we draw measures with the same total mass. Moreover, computing the Wasserstein barycenter of $X_1, \hdots, X_n$ in practice is too costly for large $n$'s, even with approximating algorithms \cite{Cuturi13,Rabin12}. To overcome these difficulties, we choose to define the central measure as the arithmetic mean of $X$, denoted by $\E(X)$, that assigns the weight $\E\left [X(A)\right ]$ to a borelian set $A$. In the point process theory, the mean measure is often referred to as the intensity function of the process. 

An easily computable estimator of this mean measure is the sample mean measure $\bar{X}_n = \left ( \sum_{i=1}^n X_i \right )/n$. We intend to build a $k$-points approximation of $\E(X)$, that is a distribution $P_{\cb}$ supported by $\cb=(c_1, \hdots, c_k)$ that approximates well $\E(X)$, based on $X_1, \hdots, X_n$. To this aim, we introduce two algorithms (batch and mini-batch) that extend classical quantization techniques intended to solve the $k$-means problem \cite{MacQueen67}. In fact, these algorithms are built to solve the $k$-means problem for $\bar{X}_n$. We prove in Section \ref{sec:Algorithms} that these algorithms provide minimax optimal estimators of a best possible $k$-points approximation of $\E(X)$, provided that $\E(X)$ satisfies some structural assumption. Interestingly, our results also prove optimality of the classical quantization techniques \cite{MacQueen67,Lloyd82} in the point sample case.

We assess the overall validity of our appproach by giving theoretical guarantees on our vectorization and clustering process in a topological data analysis framework, where measures are persistence diagrams from different shapes. In this case, vectorization via evaluation onto a fixed grid is the technique exposed in \cite{Adams17}, whereas our method has clear connections with the procedures described in \cite{Zielinski18,Royer19}.
Our theoretical results are given for vectorizations via evaluations of kernel functions around each point $c_j$, for a general class of kernel functions that encompasses the one used in \cite{Royer19}. Up to our knowledge, this provides the only theoretical guarantee on a measure-based clustering algorithm. 

At last, we perform numerical experiments in Section \ref{sec:Expes} to assess the effectiveness of our method on real and synthetic data, in a clustering and classification framework. In a nutshell, our vectorization scheme combined with standard algorithms provides state-of-the art performances on various classification and clustering problems, with a lighter computational cost. The classification problems encompass sentiment analysis of IMDB reviews \cite{imdb} as well as large-scale graph classification \cite{rozemberczki2020api,Yanardag15}. Surprisingly, our somehow coarse approach turns out to outperform more involved methods in several large-scale graph classification problems.

The paper is organized as follows: Section \ref{sec:measures_clustering_vectorization} introduces our general vectorization technique, and  conditions that guarantee a correct clustering based on it. Section \ref{sec:mean_measure_quantization} gives some detail on how our vectorization points are build, by introducing the mean measure $k$-points quantization problem. Then, two theoretically grounded algorithms are described to solve this problem from the sample $X_1, \hdots, X_n$. Section \ref{sec:appli_PD} investigates the special case where the measures are persistence diagrams built from samplings of different shapes, showing that all the previously exposed theoretical results apply in this framework. The numerical experiments are exposed in Section \ref{sec:Expes}.  Sections \ref{sec:proof_sec_clustering_based_on_vectorization} and \ref{sec:proof_sec_mean_measure_quantization} gather the main proofs of the results. Proofs of intermediate and technical results are deferred to  \Cref{sec:proof_tecsec}.  

\section{Vectorization and clustering of measures}\label{sec:measures_clustering_vectorization}
Throughout the paper we will consider finite measures on the $d$-dimensional ball $\B(0,R)$ of the Euclidean space $\R^d$, equipped with its Borel algebra. Let  $\mathcal{M}(R,M)$ denote the set of such measures of total mass smaller than $M$. For $\mu \in \mathcal{M}(R,M)$ and $f$ a borelian function from $\mathbb{R}^d$ to $\mathbb{R}$, we denote by $\int f(y) \mu(dy)$ integration of $f$ with respect to $\mu$, whenever $\int |f(y)| \mu(dy) $ is finite. We let $X$ be a random variable with values in $\mathcal{M}(R,M)$ (equipped with the smallest algebra such that $\int f dX$ is measurable, for any continuous and bounded $f$), and $X_1, \hdots, X_n$ be an i.i.d sample drawn from $X$. 

\subsection{Vectorization of measures}\label{sec:vectorization}

We aim to provide a map $v$ from $\mathcal{M}(R,M)$ into a finite-dimensional space that preserves separation between sources. That is, if we assume that there exists $(Z_1, \hdots, Z_n) \in [\![1,L]\!]^n$ a vector of (hidden) label variables such that $X_i \mid \{Z_i = \ell \} \sim X^{(\ell)}$, for different mixture components $X^{(1)}, \hdots, X^{(L)}$ taking values  in $\mathcal{M}(R,M)$, we want $\|v(X_{i_1}) - v(X_{i_2})\|$ to be large whenever $Z_{i_1} \neq Z_{i_2}$. 

    The intuition is the following. Let $\mu_1 \neq \mu_2$ be two different measures on $\B(0,R)$. Then there exists $x \in \B(0,R)$ and $r>0$ such that $\mu_1(\B(x,r)) \neq \mu_2(\B(x,r))$. Thus, the vectorization that sends $X_i$ into $X_i(\B(x,r))$ will allow to separate the sources $X^{(1)} = \delta_{\mu_1}$ and $X^{(2)} = \delta_{\mu_2}$. We extend this intuition to the multi-classes case. To adopt the quantization terminology, for $k \in \mathbb{N}^*$ we let $c_1, \hdots, c_k$ be \textit{codepoints}, and $\cb = (c_1, \hdots, c_k) \in (\R^d)^k$ is called a \textit{codebook}. For a given codebook $\cb$ and a scale $r>0$, we may represent a measure $\mu$ via the vector of weights $(\mu(\B(c_1,r)), \hdots,\mu(\B(c_k,r)))$ that encodes the mass that $\mu$ spreads around every pole $c_j$. Provided that the codepoints are discriminative (this will be discussed in the following section), separation between clusters of measure will be preserved. In practice, convolution with kernels is often preferred to local masses (see, e.g., \cite{Royer19}). To ease computation, we will restrict ourselves to the following class of kernel functions. 
\begin{defi}\label{def:kernel_functions}
For $(p,\delta) \in \mathbb{N}^*\times [0,1/2]$, a function $\psi: \mathbb{R}^+ \rightarrow \mathbb{R}^+$ is called a $(p,\delta)$-kernel function if
\[
\begin{array}{@{}clcll}
i)& \|\psi\|_\infty \leq 1, & \qquad & ii)& \sup_{|u| \leq 1/p} \psi(u) \geq 1- \delta, \\
iii) &  \sup_{|u| >2p} \psi(u) \leq \delta, & \qquad & iv)& \mbox{$\psi$ is $1$-Lipschitz}.
\end{array}
\]
\end{defi}
Note that a $(p,\delta)$-kernel is also a $(q,\delta)$-kernel, for $q >p$. This definition of a kernel function encompasses widely used kernels, such as Gaussian or Laplace kernels. In particular, the function $\psi(u) = exp(-u)$ that is used in \cite{Royer19} is a  $(p,1/p)$-kernel for $p \in \mathbb{N}^*$. The $1$-Lispchitz requirement is not necessary to prove that the representations of two separated measures will be well-separated. However, it is a key assumption to prove that the representations of two measures from the same cluster will remain close in $\R^k$. From a theoretical viewpoint, a convenient kernel is $\psi_0:x \mapsto (1- ((x-1) \vee 0)) \vee 0$, which is a $(1,0)$-kernel, thus a $(p,0)$-kernel for all $p \in \mathbb{N}^*$. 

From now on we assume that the kernel $\psi$ is fixed, and, for a $k$-points codebook $\cb$ and scale factor $\sigma$, consider the vectorization
\begin{align}\label{eq:vectorization}
v_{\cb,\sigma}: \left \{ \begin{array}{@{}ccc}
\mathcal{M}(R,M) & \rightarrow & [0,M]^k \\
\mu & \mapsto & \left ( \int \psi(\|u-c_1\|/\sigma) \mu(du), \hdots, \int \psi(\|u-c_k\|/\sigma) \mu(du) \right ).
\end{array}
\right .
\end{align}
Note that the dimension of the vectorization depends on the cardinality of the codebook $\cb$. To guarantee that such a vectorization is appropriate for a clustering purpose is the aim of the following section. 

\subsection{Discriminative codebooks and clustering of measures}\label{sec:measures_clustering}

In this section we assume that there exists a vector $(Z_1, \hdots, Z_n) \in [\![1,L]\!]^n$ of (hidden) label variables, and let $M_1, \hdots, M_L$ be such that, if $Z_i = \ell$, $X_i \in \mathcal{M}(R,M_\ell)$. We also denote $M = \max_{\ell \leq L} M_\ell$, and investigate under which conditions the vectorization exposed in the above section provides a representation that is suitable for clustering.    For a given codebook $\cb$, we introduce the following definition of $(p,r,\Delta)$-scattering to quantify how well $\cb$ will allow to separate clusters.

\begin{defi}\label{def:scatter}
Let $(p,r,\Delta) \in (\mathbb{N}^* \times \R^+ \times \R^+)$. A codebook $\cb \in \B(0,R)^k$ is said to $(p,r,\Delta)$ -\textbf{shatter} $X_1, \hdots, X_n$ if, for any $i_1, i_2$ $\in$ $[\![1,n]\!]$ such that $Z_{i_1} \neq Z_{i_2}$, there exists $j_{i_1,i_2} \in [\![1,k]\!]$ such that
\begin{align*}
X_{i_1}(\B(c_{j_{i_1,i_2}},r/p)) & \geq  X_{i_2}(\B(c_{j_{i_1,i_2}},4pr)) + \Delta, \quad \mbox{or} \\
X_{i_2}(\B(c_{j_{i_1,i_2}},r/p)) & \geq X_{i_1}(\B(c_{j_{i_1,i_2}},4pr)) + \Delta.
\end{align*}
\end{defi}
In a nutshell, the codebook $\cb$ shatters the sample if two different measures from two different clusters have different masses around one of the codepoint of $\cb$, at scale $r$. Note that, for any $i$, $j$, $X_{i}(\B(c_j,r/p)) \geq X_i(\{c_j\})$, so that a stronger definition of shattering in terms of $X_i(\{c_j\})$'s might be stated, in the particular case where $X_i(\{c_j\}) >0$. The following Proposition ensures that a codebook which shatters the sample yields a vectorization into separated clusters, provided the kernel decreases fast enough.
\begin{prop}\label{prop:vectorization_scattering}
Assume that $\cb \in \B(0,R)^k$ shatters $X_1, \hdots, X_n$, with parameters $(p,r,\Delta)$. Then, if $\Psi$ is a $(p,\delta)$-kernel, with $\delta \leq \frac{\Delta}{4M}$, we have, for all $i_1$, $i_2$ $\in$ $[\![1,n]\!]$ and $\sigma \in [r,2r]$, 
\begin{align*}
Z_{i_1} \neq Z_{i_2} \quad \Rightarrow \quad \|v_{\cb,\sigma}(X_{i_1}) - v_{\cb,\sigma}(X_{i_2})\|_\infty \geq \frac{\Delta}{2}.
\end{align*}
\end{prop}
A proof of Proposition \ref{prop:vectorization_scattering} can be found in Section \ref{sec:proof_prop_vectorization_scattering}. This proposition sheds some light on how $X_1, \hdots, X_n$ has to be shattered with respect to the parameters of $\Psi$. Indeed, assume that $\Delta=1$ (that is the case if the $X_i$'s are integer-valued measures, such as count processes for instance). Then, to separate clusters, one has to choose $\delta$ small enough compared to $1/M$, and thus $p$ large enough if $\Psi$ is non-increasing. Hence, the vectorization will work roughly if the support points of two different counting processes are $rp$-separated, for some scale $r$. This scale $r$ then drives the choice of the bandwidth $\sigma$. Section \ref{sec:appli_PD} below provides an instance of shattered measures in the case where measures are persistence diagrams originating from well separated shapes. If the requirements of Proposition \ref{prop:vectorization_scattering} are fulfilled, then a standard hierarchical clustering procedure such as Single Linkage with $L_\infty$ distance will separate the clusters for the scales smaller than $\Delta/2$. Note that, without more assumptions, this global clustering scheme can result in more than $L$ clusters, nested into the ground truth label classes. 

Now, to achieve a perfect clustering of the sample based on our vectorization scheme, we have to ensure that measures from the same cluster are not too far in terms of Wasserstein distance, implying in particular that they have the same total mass. This motivates the following definition. For $p\in [\![1,+\infty ]\!]$ and $\mu_1$, $\mu_2$ $\in \mathcal{M}(R,M)$ such that $\mu_1(\R^d) = \mu_2(\R^d)$, we denote by $W_p(\mu_1,\mu_2)$ the $p$-Wasserstein distance between $\mu_1$ and $\mu_2$.
\begin{defi}\label{def:contrated}
The sample of measures $X_1, \hdots X_n$ is called $w$-concentrated if, for all $i_1$, $i_2$ in $[\![1,n]\!]$ such that $Z_{i_1} = Z_{i_2}$,
\[
\begin{array}{@{}clccl}
i) & X_{i_1}(\mathbb{R}^d) = X_{i_2}(\mathbb{R}^d), & \qquad & ii) & W_1(X_{i_1},X_{i_2}) \leq w.
\end{array}
\]
\end{defi}
It now falls under the intuition that well-concentrated and shattered sample measures are likely to be represented in $\R^k$ by well-clusterable points. A precise statement is given by the following Proposition \ref{prop:well_clustered}.
\begin{prop}\label{prop:well_clustered}
Assume that $X_1, \hdots, X_n$ is $w$-concentrated. If $\Psi$ is $1$-Lipschitz, then, for all $\cb \in \B(0,R)^k$ and $\sigma >0$, for all $i_1$, $i_2$ in $[\![1,n]\!]$ such that $Z_{i_1} = Z_{i_2}$,
\begin{align*}
\| v_{\cb,\sigma}(X_{i_1}) - v_{\cb,\sigma}(X_{i_2}) \|_\infty \leq \frac{w}{\sigma}.
\end{align*}
Therefore, if $X_1, \hdots, X_n$ is $(p,r,\Delta)$-shattered by $\cb$, and $(r\Delta/4)$-concentrated, then, for any $(p,\delta)$-kernel satisfying $\delta \leq \frac{\Delta}{4M}$, we have, for $\sigma \in [r,2r]$,
\[
\begin{array}{@{}ccc}
Z_{i_1}=Z_{i_2} & \Rightarrow & \| v_{\cb,\sigma}(X_{i_1}) - v_{\cb,\sigma}(X_{i_2}) \|_\infty \leq \frac{\Delta}{4}, \\
Z_{i_1} \neq Z_{i_2} & \Rightarrow & \| v_{\cb,\sigma}(X_{i_1}) - v_{\cb,\sigma}(X_{i_2}) \|_\infty \geq \frac{\Delta}{2}.
\end{array}
\]
\end{prop}
A proof of Proposition \ref{prop:well_clustered} is given in Section \ref{sec:proof_prop_well_clustered}. An immediate consequence of Proposition \ref{prop:well_clustered} is that $(p,r,\Delta)$-shattered and $r\Delta/4$-concentrated sample measures can be vectorized in $\R^k$ into a point cloud that is structured in $L$ clusters. These clusters can be exactly recovered via Single Linkage clustering, with stopping parameter in $]\Delta/4, \Delta/2]$. In practice, tuning the parameter $\sigma$ is crucial. Some heuristic is proposed in \cite{Royer19} in the special case of i.i.d persistence diagrams. An alternative calibration strategy is proposed in Section \ref{sec:appli_PD}. Constructing shattering codebooks from sample is the purpose of the following section. 
          
\section{Quantization of the mean measure}\label{sec:mean_measure_quantization} 

To find discriminative codepoints from a measure sample $X_1, \hdots, X_n$, a naive approach could be to choose a $\varepsilon$-grid of $\B(0,R)$, for $\varepsilon$ small enough. Such an approach ignores the sample, and results in quite large codebooks, with $k \sim (R/\varepsilon)^{d}$ codepoints. An opposite approach is to sample $k_i$ points from every $X_i$, with $k=\sum_{i=1}^n k_i$, resulting in a $k$-points codebook that might miss some discriminative areas, especially in the case where discriminative areas have small $X_i$-masses. We propose an intermediate procedure, based on quantization of the mean measure. Definition \ref{def:mean_measure} below introduces the mean measure.  
\begin{defi}\label{def:mean_measure} Let $\mathcal{B}(\R^d)$ denote the borelian sets of $\R^d$. The \textit{mean measure} $\mathbb{E}(X)$ is defined as the  measure such that 
\[
\forall A \in \mathcal{B}(\R^d) \quad  \E(X)(A) = \mathbb{E} \left( X(A) \right ).
\]
The \textit{empirical mean measure} $\bar{X}_n$ may be defined via  $\bar{X}_n(A) = \frac{1}{n} \sum_{i=1}^n X_i(A)$.
\end{defi}
In the case where the measures of interest are persistence diagrams, the mean measure defined above is  the expected persistence diagram, defined in \cite{ChazalDivol18}. If the sample measures are point processes, $\E(X)$ is the intensity function of the process. It is straightforward that, if $\mathbb{P} \left ( X \in \mathcal{M}(R,M) \right )=1$, then both $\E(X)$ and $\bar{X}_n$ are (almost surely) elements of $\mathcal{M}(R,M)$. 

For a codebook $\cb = (c_1, \hdots, c_k) \in \mathcal{B}(0,R)^k$, we let 
\begin{align*}
W_j(\cb)& = \{ x \in \mathbb{R}^d \mid  \forall i < j \quad \| x- c_j\| < \| x-c_i\| \quad \mbox{and} \quad \forall{i >j} \quad \| x- c_j\| \leq \|x-c_i\| \}, \\
N(\cb) & = \{x \mid \exists i < j \quad x \in W_i(\cb) \quad \mbox{and} \quad \|x-c_i\| = \|x-c_j\| \}, 
\end{align*}
so that $(W_1(\cb), \hdots, W_k(\cb))$ forms a partition of $\R^d$ and $N(\cb)$ represents the skeleton of the Voronoi diagram associated with $\cb$. To grid the support of $\E(X)$, we want to minimize the distortion $F$, that is 
\begin{align*}
F(\cb) = \int  \min_{j=1, \hdots, k} \|u-c_j\|^2 \E(X)(du) = W_2^2 (\E(X),P_\cb),
\end{align*}
where $P_\cb = \sum_{j=1}^k \E(X)(W_j(\cb)) \delta_{c_j}$ is a measure supported by $\cb$. It is worth noting that if $P'_k = \sum_{j=1}^k \mu_j \delta_{c_j}$, with $\sum_{j=1}^{k} \mu_j = \E(X)(\R^d)$, then $W_2^2(\E(X),P'_k)\geq W_2^2 (\E(X),P_\cb)$, so that $P_\cb$ is the best approximation of $\E(X)$ with support $\cb$ in terms of $W_2$ distance.

According to \cite[Corollary 3.1]{Fischer10}, since $\E(X) \in \mathcal{M}(R,M)$, there exist minimizers $\cb^*$ of $F(\cb)$, and we let $\mathcal{C}_{opt}$ denote the set of such minimizers. As well, if $\mathcal{M}_k(R,M)$ denotes the subset of $\mathcal{M}(R,M)$ that consists of distributions supported by $k$ points and $\cb^* \in \mathcal{C}_{opt}$, then $P_{\cb^*}$ is a best approximation of $\E(X)$ in $\mathcal{M}_k(R,M)$ in terms of $W_2$ distance. Note that distortions and thus optimal codebooks may be defined for $W_p$, $p\neq 2$. However, simple algorithms are available to approximate such a $\cb^*$ in the special case $p=2$, as discussed below.

\subsection{Batch and mini-batch algorithms}\label{sec:Algorithms}

This section exposes two algorithms that are intended to approximate a best $k$-points codebook for $\E(X)$ from a measure sample $X_1, \hdots, X_n$. These algorithms are extensions of two well-known clustering algorithms, namely the Lloyd algorithm (\cite{Lloyd82}) and Mac Queen algorithm (\cite{MacQueen67}). We first introduce the counterpart of Lloyd algorithm.

\bigskip

\begin{lstlisting}[frame=single,caption={Batch algorithm (Lloyd)},label={lst:lloyd},abovecaptionskip=-\medskipamount]
Input:%*$X_1, \hdots, X_n$ and $k$ *);
# Initialization
%* Sample $c^{(0)}_1$, $c^{(0)}_2$,\ldots $c^{(0)}_k$ from $\bar{X}_n$.*);
while %*$\cb^{(t+1)} \neq \cb^{(t)}$*) do:
  # Centroid update.
  for j in 1..%*$k$*):
    %*$c^{(t+1)}_j= \frac{1}{\bar{X}_n(W_j(\cb^{(t)}))}  \int u \1_{W_j(\cb^{(t)})}(u) \bar{X}_n(du) $*);
Output:%*$\cb^{(T)}$ (codebook of the last iteration) *).
\end{lstlisting}

~
Note that Algorithm \ref{lst:lloyd} is a batch algorithm, in the sense that every iteration needs to process the whole data set $X_1, \hdots, X_n$. Fortunately, Theorem \ref{thm:optLloyd} below ensures that a limited number of iterations are required for Algorithm \ref{lst:lloyd} to provide an almost optimal solution. In the sample point case, that is when we observe $n$ i.i.d points $X^{(1)}_i$, Algorithm \ref{lst:lloyd} is the usual Lloyd's algorithm. In this case, the mean measure $\mathbb{E}(X)$ is the distribution of $X_1^{(1)}$, that is the usual sampling distribution of the $n$ i.i.d points. 
As well, the counterpart of Mac-Queen algorithm \cite{MacQueen67} for standard $k$-means clustering is the following mini-batch algorithm. We let $\pi_{\B(0,R)^k}$ denote the projection onto $\B(0,R)^k$. 

\medskip

\begin{lstlisting}[frame=single,caption={Mini-batch algorithm (Mac-Queen)},label={lst:macqueen},abovecaptionskip=-\medskipamount]
Input:%*$X_1, \hdots, X_n$, divided into mini-batches $(B_1, \hdots, B_{T})$ of sizes $(n_1, \hdots, n_T)$, and $k$. For $j = 1, \hdots, T$, $B_j$ is divided in two halves, $B_j^{(1)}$ and $B_j^{(2)}$. Maximal radius R  *);
# Initialization
%* Sample $c^{(0)}_1$, $c^{(0)}_2$,\ldots $c^{(0)}_k$ from $\bar{X}_n$.*);
for %*$j=0, \hdots, T-1$*) do:
  # Centroid update.
  for j in 1..%*$k$*):
    %*$c^{(t+1)}_j= \pi_{\B(0,R)^k} \left ( c_j^{(t)} - \frac{  \int  (c_j^{(t)} - u) \1_{W_j(\cb^{(t)})}(u) \bar{X}_{B^{(2)}_{t+1}}(du) }{(t+1)\bar{X}_{B_{t+1}^{(1)}}  \left ( W_j(\cb^{(t)})  \right )}  \right )$*);
Output:%*$\cb^{(T)}$ (codebook of the last iteration) *).
\end{lstlisting}

Note that the mini-batches split in two halves is motivated by technical considerations only that are exposed in Section \ref{sec:proof_thm_opt_Mc_Queen}.  
Whenever $n_i=1$ for $i=1, \hdots, n$ (and $B_i^{(1)} = B_i^{(2)} = \{X_i\}$), Algorithm \ref{lst:macqueen} is a slight modification of the original Mac-Queen algorithm \cite{MacQueen67}. Indeed, the Mac-Queen algorithm takes mini-batches of size $1$, and estimates the population of the cell $j$ at the $t$-th iteration via $\sum_{\ell=1}^t \hat{p}^{(\ell)}_j$ instead of $t \hat{p}_j^{(t)}$, where $\hat{p}_j^{(t)} = \bar{X}_{B^{(1)}_t}  \left ( W_j(\cb^{(t)})  \right )$. These modifications are motivated by Theorem \ref{thm:opt_Mc_Queen}, that guarantees near-optimality of the output of Algorithm \ref{lst:macqueen}, provided that the mini-batches are large enough. 

\subsection{Theoretical guarantees}

We now investigate the distortions of the outputs of Algorithm \ref{lst:lloyd} and \ref{lst:macqueen}.  In what follows, $F^*$ will denote the optimal distortion achievable with $k$ points, that is $F^* = F(\cb)$, where $\cb \in \mathcal{C}_{opt}$. Note that $F^*=0$ if and only if $\E(X)$ is supported by less than $k$ points, in which case the quantization problem is trivial. In what follows we assume that $\E(X)$ is supported by more than $k$ points. Basic properties of $\mathcal{C}_{opt}$ are recalled below.
\begin{prop}{\cite[Proposition 1]{Levrard18}}\label{prop:B_p_min} Recall that $\mathbb{E}(X) \in \mathcal{M}(R,M)$. Then,
\begin{enumerate}
\item $ B = \inf_{\cb^* \in \mathcal{C}_{opt}, j \neq i}{\|c_i^* - c_j^*\|} >0$,
\item $p_{min} = \inf_{\cb^* \in \mathcal{C}_{opt}, j=1, \hdots, k}{\E(X) \left ( W_j(\cb^*) \right )} >0$.
\end{enumerate}
\end{prop}
We will further assume that $\E(X)$ satisfies a so-called \textit{margin condition}, defined in \cite[Definition 2.1]{Levrard15} and recalled below. For any subset $A$ of an Euclidean space, and $r>0$, we denote by $\B(A,r)$ the set $\cup_{u \in A} \B(u,r)$.
   \begin{defi}\label{def:margincondition}
         $\E(X) \in \mathcal{M}(R,M)$ satisfies a margin condition with radius $r_0 >0$ if and only if, for all $0 \leq t \leq r_0$,
					\begin{align*}
					 \sup_{\cb^* \in \mathcal{C}_{opt}} \E(X) \left ( \mathcal{B}(N(\cb^*),t) \right )  \leq \frac{B p_{min}}{128 R^2}t.
					\end{align*}
         \end{defi}
In a nutshell, a margin condition ensures that the mean distribution $\E(X)$ is well-concentrated around $k$ poles. For instance, finitely-supported distributions satisfy a margin condition. 
Following \cite{Levrard18}, a margin condition will ensure that usual $k$-means type  algorithms are almost optimal in terms of distortion. Up to our knowledge, margin-like conditions are always required to guarantee convergence of Lloyd-type algorithms \cite{Monteleoni16,Levrard18}. 

 A special interest will be paid to the sample-size dependency of the excess distortion. From this standpoint, a first negative result may be derived from \cite[Proposition 7]{Levrard18}.  Namely, for any empirically designed codebook $\hat{\cb}$, we have
\begin{align}\label{eq:borneinf}
\sup_{\{X \mid \E(X)\mbox{ has a $r_0$-margin}\}} \frac{\E( F(\hat{\cb}) - F^*)}{\E(X)(\R^d)} \geq c_0 R^2 \frac{k^{ 1 - \frac{2}{d}}}{n}. 
\end{align}
In fact, this bounds holds in the special case where $X$ satisfies the additional assumption $X=\delta_{X^{(1)}}$ a.s., pertaining to the vector quantization case. Thus it holds in the general case. This small result ensures that the sample-size dependency of the minimax excess distortion over the class of distribution of measures whose mean measure satisfies a margin condition with radius $r_0$ is of order $1/n$ or greater. 

A first upper bound on this minimax excess distortion may be derived from the following Theorem \ref{thm:optLloyd}, that investigates the performance of the output of Algorithm \ref{lst:lloyd}. We recall that, for $N \in \mathbb{N}^*$, $\mathcal{M}_N(R,M)$ denotes the subset of measures in $\mathcal{M}(R,M)$ that are supported by $N$ points.
\begin{theorem}\label{thm:optLloyd}
Let $X \in \mathcal{M}_{N_{max}}(R,M)$, for some $N_{max} \in \mathbb{N}^*$. Assume that $\mathbb{E}(X)$ satisfies a margin condition with radius $r_0$, and denote by $R_0= \frac{Br_0}{16\sqrt{2}R}$, $\kappa_0 = \frac{R_0}{R}$. Choose $T = \lceil \frac{\log(n)}{\log(4/3)} \rceil$, and let $\cb^{(T)}$ denote the output of Algorithm \ref{lst:lloyd}. If $\cb^{(0)} \in \B(\mathcal{C}_{opt},R_0)$, then, for $n$ large enough, with probability $1-9e^{-c_1 n p_{min}^2 \kappa_0^2/M^2 } - e^{-x}$, where $c_1$ is a constant, we have
\begin{align*}
F(\cb^{(T)}) - F^* \leq \E(X)(\R^d) \left (\frac{ B^2r_0^2}{512 R^2 n} + C \frac{M^2 R^2 k^2d \log(k)}{n p_{min}^2}(1+x) \right ),
\end{align*}
for all $x>0$, where $C$ is a constant.
\end{theorem}
A proof of Theorem \ref{thm:optLloyd} is given in Section \ref{sec:proof_thm_optLloyd}. Combined with \eqref{eq:borneinf}, Theorem \ref{thm:optLloyd} ensures that Algorithm \ref{lst:lloyd} reaches the minimax precision rate in terms of excess distortion after $O(\log(n))$ iterations, provided that the initialization is good enough. Note that Theorem \ref{thm:optLloyd} is valid for discrete measures that are supported by a uniformly bounded number of points $N_{max}$. This assumption is relaxed for Algorithm \ref{lst:macqueen} in Theorem \ref{thm:opt_Mc_Queen}.

In the vector quantization case, Theorem \ref{thm:optLloyd} might be compared with \cite[Theorem 3.1]{Levrard15} for instance. In this case, the dependency on the dimension $d$ provided by Theorem \ref{thm:optLloyd} is sub-optimal. Slightly anticipating, dimension-free bounds in the mean-measure quantization case exist, for instance by considering the output of Algorithm \ref{lst:macqueen}.

Theorem \ref{thm:optLloyd} guarantees that choosing $T=2 \log(n)$ and repeating several Lloyd algorithms starting from different initializations provides an optimal quantization scheme. From a theoretical viewpoint, for a fixed $R_0$, the number of trials to get a good initialization grows exponentially with the dimension $d$. In fact, Algorithm \ref{lst:lloyd} seems quite sensitive to the initialization step, as for every EM-like algorithm. From a practical viewpoint, we use $10$ initializations using a \texttt{$k$-means++} sampling \cite{Arthur07} for the experiments of Section \ref{sec:Expes}, leading to results that are close to the reported ones for Algorithm \ref{lst:macqueen}. 

 Note that combining \cite[Theorem 3]{Levrard18} or \cite{Monteleoni16} and a deviation inequality for distortions such as in \cite{Levrard15} gives an alternative proof of the optimality of Lloyd type schemes, in the sample points case where $X_i=\delta_{X_i^{(1)}}$. In addition, Theorem \ref{thm:optLloyd} provides an upper bound on the number of iterations needed, as well as an extension of these results to the quantization of mean measure case. Its proof, that may be found in Section \ref{sec:proof_thm_optLloyd}, relies on stochastic gradient techniques in the convex and non-smooth case.
Bounds for the single-pass Algorithm \ref{lst:macqueen} might be stated the same way.
\begin{theorem}\label{thm:opt_Mc_Queen}
Let $X \in \mathcal{M}(R,M)$. Assume that $\E(X)$ satisfies a margin condition with radius $r_0$, and denote by $R_0= \frac{Br_0}{16\sqrt{2}R}$, $\kappa_0 = R_0/R$. If $(B_1, \hdots, B_T)$ are equally sized mini-batches of length $c k M^2\log(n)/(\kappa_0 p_{min})^2$, where $c$ is a positive constant, and $\cb^{(T)}$ denotes the output of Algorithm \ref{lst:macqueen}, then, provided that $\cb^{(0)} \in \mathcal{B}(\mathcal{C}_{opt},R_0)$, we have
\begin{align*}
\mathbb{E} \left ( F(\cb^{(T)}) - F^* \right ) \leq \E(X)(\R^d) \left ( C k^2 M^3 R^2 \frac{\log(n)}{n \kappa_0^2 p_{min}^3} \right ).
\end{align*}

In the sample point case, the same result holds with the centroid update
\[
c^{(t+1)}_j=  c_j^{(t)} - \frac{ \int (c_j^{(t)} - u) \1_{W_j(\cb^{(t)})}(u) \bar{X}_{B_{t+1}}(du)}{(t+1)\bar{X}_{B_{t+1}}  \left ( W_j(\cb^{(t)})  \right )} ,
\]
that is without splitting the batches.
\end{theorem} 
A proof of Theorem \ref{thm:opt_Mc_Queen} is given in Section \ref{sec:proof_thm_opt_Mc_Queen}. Note that Theorem \ref{thm:opt_Mc_Queen} does not require the values of $X$ to be finitely supported measures, contrary to Theorem \ref{thm:optLloyd}.  Theorem \ref{thm:opt_Mc_Queen} entails that the resulting codebook of Algorithm \ref{lst:macqueen} has an optimal distortion, up to a $\log(n)$ factor and provided that a good enough initialization is chosen. As for Algorithm \ref{lst:lloyd}, several initializations may be tried and the codebook with the best empirical distortion is chosen. In practice, the results of Section \ref{sec:Expes} are given for a single initialization at random using \texttt{$k$-means++} sampling \cite{Arthur07}.

Note that Theorem \ref{thm:opt_Mc_Queen} provides a bound on the expectation of the distortion. Crude deviation bounds can be obtained using for instance a bounded difference inequality (see, e.g., \cite[Theorem 6.2]{Massart13}). In the point sample case, more refined bounds can be obtained, using for instance \cite[Theorem 4.1, Proposition 4.1]{Levrard15}. To investigate whether these kind of bounds still hold in the measure sample case is beyond the scope of the paper. Note also that the bound on the excess distortion provided by Theorem \ref{thm:opt_Mc_Queen} does not depend on the dimension $d$. This is also the case in \cite[Theorem 3.1]{Levrard15}, where a dimension-free theoretical bound on the excess distortion of an empirical risk minimizer is stated in the sample points case. Interestingly, this bound also has the correct dependency in $n$, namely $1/n$. According to Theorem \ref{thm:optLloyd} and \ref{thm:opt_Mc_Queen}, providing a quantization scheme that provably achieves a dimension-free excess distortion of order $1/n$ in the sample measure case remains an open question. 

We may now connect Theorem \ref{thm:optLloyd} and \ref{thm:opt_Mc_Queen} to the results of Section \ref{sec:measures_clustering}.
\begin{corollary}\label{cor:vectorization_Lloyd}
Assume that $\E(X)$ satisfies the requirements of Theorem \ref{thm:optLloyd}, and that $\cb^*$ provides a $(p,r,\Delta)$ shattering of $X_1, \hdots, X_n$, with $p\geq 2$. Let $\hat{\cb}_n$ denote the output of Algorithm \ref{lst:lloyd}. Then $\hat{\cb}_n$ is a $(\lfloor \frac{p}{2} \rfloor ,r,\Delta)$ shattering of $X_1, \hdots, X_n$, with probability larger than $1-\exp{\left [-C\left ( \frac{n r^2 p_{min}^2}{p^2M^2R^2k^2 d \log(k)} - \frac{p_{min}^2 B^2 r_0^2}{M^2 R^4 k^2 d \log(k)} \right ) \right ]}$, where $C$ is a constant.
\end{corollary}
A proof of Corollary \ref{cor:vectorization_Lloyd} is given in Section \ref{sec:proof_cor_vectorization_Lloyd}. Corollary \ref{cor:vectorization_Lloyd} ensures that if $\cb^*$ shatters well $X_1, \hdots, X_n$, so will $\hat{\cb}_n$, where $\hat{\cb}_n$ is built with Algorithm \ref{lst:lloyd}. To fully assess the relevance of our vectorization technique, it remains to prove that $k$-points optimal codebooks for the mean measure provide a shattering of the sample measure, with high probability. This kind of result implies more structural assumptions on the components of the mixture $X$. The following section describes such a situation.

\subsection{Application: clustering persistence diagrams}\label{sec:appli_PD}

In this section we investigate the properties of our measure-clustering scheme in a particular instance of i.i.d. measure observations. Recall that in our mixture model, $(Z_1, \hdots, Z_n)$ is the vector of i.i.d. labels such that $X_i\mid \{Z_i = \ell\} \sim X^{(\ell)}$. We further denote by $\pi_\ell = \mathbb{P}(Z_i=\ell)$, so that $X \sim \sum_{\ell=1}^L \pi_\ell X^{(\ell)}$.

Our mixture of persistence diagrams model is the following. For $\ell \in [\![1, L]\!]$ let $S^{(\ell)}$ denote a compact $d_\ell$-manifold and $D^{(\ell)}$ denote the persistence diagram generated by $d_{S^{(\ell)}}$, the distance to $S^{(\ell)}$. This persistence diagram can either be  thought of as a multiset of points (set of points with integer-valued multiplicity) in the half-plane $H^+=\{(b,d) \in \R^2\mid 0 \leq b \leq  d \}$ (for a general introduction to persistence diagrams the reader is referred to \cite[Section 11.5]{boissonnat2018geometric}), or as a discrete measure on $H^+$ (see, e.g., \cite{ChazalDivol18} for more details on the measure point of view on persistence diagrams). For a fixed scale $s>0$, we define $D^{(\ell)}_{\geq s}$ as the thresholded persistence diagram of $d_{S^{(\ell)}}$, that is
\[
D^{(\ell)}_{\geq s} = \sum_{\{(b,d) \in D^{(\ell)} \mid d-b \geq s\}}n(b,d)\delta_{(b,d)} := \sum_{j=1}^{k_0^{(\ell)}} n(m_j^{(\ell)}) \delta_{m_j^{(\ell)}},
\] 
where the multiplicities $n(b,d) \in \mathbb{N}^*$, and the $m_j^{(\ell)}$'s satisfy $(m_j^{(\ell)})_2 - (m_j^{(\ell)})_1 \geq s$. The following lemma ensures that $k_0^{(\ell)}$ is finite.
\begin{lem} \label{lemma-truncated-dgm}
Let $S$ be a compact subset of $\R^d$, and $D$ denote the persistence diagram of the distance function $d_S$. For any $s >0$, the truncated diagram consisting of the points $m=(m^1, m^2) \in D$ such that $m^2 - m^1 \geq s$ is finite.    
\end{lem}
      
A proof of Lemma \ref{lemma-truncated-dgm} is given in Appendix \ref{sec:proof_lemma_truncated_dgm}. 
Now, for $\ell \in [\![1,L]\!]$ we let $\mathbb{Y}_{N_\ell}$ be an $N_\ell$-sample drawn on $S^{(\ell)}$ with density $f^{(\ell)}$ satisfying $f^{(\ell)}(u) \geq f_{min,\ell}$, for $u \in S^{(\ell)}$. If $\hat{D}^{(\ell)}$ denotes the persistence diagram generated by $d_{\mathbb{Y}_{N_\ell}}$, we define $X^{(\ell)}$ as a thresholded version of $\hat{D}^{(\ell)}$, that is
\[
X^{(\ell)} \sim \hat{D}^{(\ell)}_{\geq s-h_\ell},
\]
for some bandwidth $h_\ell$. To sum up, conditionally on $Z_i=\ell$, $X_i$ is a thresholded persistence diagram corresponding to a $N_\ell$-sampling of the shape $S^{(\ell)}$.

To retrieve the labels $Z_i$'s from a vectorization of the $X_i$'s, we have to assume that the persistence diagrams of the underlying shapes differ by at least one point.
\begin{defi}\label{def:separability_classes}
The shapes $S^{(1)}, \hdots, S^{(\ell)}$ are \textbf{discriminable at scale} $\mathbf{s}$ if for any $1 \leq \ell_1 < \ell_2 \leq L$ there exists $m_{\ell_1,\ell_2} \in H^+$ such that
\[
D^{(\ell_1)}_{\geq s}(\{m_{\ell_1,\ell_2}\}) \neq D^{(\ell_2)}_{\geq s}(\{m_{\ell_1,\ell_2}\}),
\]
where the thresholded persistence diagrams are considered as measures.
\end{defi}
Note that if $m_{\ell_1,\ell_2}$ satisfies the discrimination condition stated above, then $m_{\ell_1,\ell_2} \in D^{(\ell_1)}_{\geq s}$ or $m_{\ell_1,\ell_2} \in D^{(\ell_2)}_{\geq s}$. Next, we must ensure that optimal codebooks of the mean measure have codepoints close enough to discrimination points. 
\begin{prop}\label{prop:mean_meas_closeto_disc_points}
Let $h_\ell = \left ( \frac{C_{d_\ell} \log(N_\ell)}{f_{min,\ell} N_\ell} \right)^{1/d_\ell}$, for some constants $C_{d_\ell}$, and $h = \max_{\ell \leq L} h_\ell$. Moreover, let $M_\ell = D^{(\ell)}_{\geq s}(H^+)$, $\bar{M}= \sum_{\ell = 1}^L \pi_\ell M_\ell$, and $\pi_{min} = \min_{ \ell \leq L} \pi_\ell$. 

Assume that $S^{(1)}, \hdots, S^{(L)}$ are discriminable at scale $s$, and let $m_1, \hdots, m_{k_0}$ denote the discrimination points. Let $K_0(h)$ denote
\[
\inf\{ k \geq 0 \mid \exists t_1, \hdots, t_{k} \quad  \bigcup_{ \ell =1}^{L} D_{\geq s}^{(\ell)} \setminus \{m_1, \hdots, m_{k_0} \}  \subset \bigcup_{s=1}^{k} \B_{\infty}(t_s,h) \}.
\]
Let $k \geq k_0 + K_0(h)$, and $(c^*_1, \hdots, c^*_k)$ denote an optimal $k$-points quantizer  of $\E(X)$. Then, provided that $N_\ell$ is large enough for all $\ell$, we have
\[
\forall j \in [\![1,k_0]\!] \quad \exists p \in [\![1,k]\!] \quad \| c_p^* - m_j \|_\infty \leq \frac{5\sqrt{\bar{M}}h}{\sqrt{\pi_{min}}}.
\]
\end{prop}
The proof of Proposition \ref{prop:mean_meas_closeto_disc_points} is given in Section \ref{sec:proof_prop_meas_closeto_disc_points}. If $\bar{D}_{\geq s}$ denotes the mean persistence diagram $\sum_{\ell = 1}^{L} \pi_\ell D^{(\ell)}_{\geq s}$, and $\bar{D}_{\geq s}$ has $K_0$ points, then it is immediate that $k_0 + K_0(h) \leq K_0$. Moreover, we also have $k_0 \leq \frac{L(L+1)}{2}$. Proposition \ref{prop:mean_meas_closeto_disc_points} ensures that the discrimination points are well enough approximated by optimal $k$-centers of the expected persistence diagram $\E(X)$, provided the shapes $S^{(\ell)}$ are well-enough sampled and $k$ is large enough so that $\bar{D}_{\geq s}$ is well-covered by $k$ balls with radius $h$. Note that this is always the case if we choose $k=K_0$, but also allows for smaller $k$'s.

In turn, provided that the shapes $S^{(1)}, \hdots, S^{(L)}$ are discriminable at scale $s$ and that $k$ is large enough, we can prove that an optimal $k$-points codebook $\cb^*$ is a $(p,r,\Delta)$-shattering of the sample, with high probability.
\begin{prop}\label{prop:PD_vectorization_shattering}
Assume that the requirements of Proposition \ref{prop:mean_meas_closeto_disc_points} are satisfied. Let $\tilde{B}= \min_{i=1, \hdots, k_0, j =1, \hdots, K_0, j \neq i} \| m_i - m_j \|_\infty \wedge s$. Let $\kappa >0$ be a small enough constant. Then, if $N_\ell$ is large enough for all $\ell \in [\![1,\ell]\!]$, $X_1, \hdots, X_n$ is $(p,r,1)$-shattered by $\cb^*$, with probability larger than $1- n \max_{\ell \leq L}N_\ell^{-\left ( \frac{(\kappa \tilde{B})^{d_\ell} f_{min,\ell} N_\ell}{C_\ell d_\ell \log(N_\ell)}\right)}$, provided that
\begin{itemize}
\item $\frac{r}{p} \geq 2 \kappa \tilde{B}$,
\item $4rp \leq \left( \frac{1}{2} - \kappa \right ) \tilde{B}$.
\end{itemize}
Moreover, on this probability event, $X_1, \hdots, X_n$ is $2 M \kappa \tilde{B}$-concentrated.  
\end{prop}
A proof of Proposition \ref{prop:PD_vectorization_shattering} is given in Section \ref{sec:proof_prop_PD_vectorization_shattering}. In turn, Proposition \ref{prop:PD_vectorization_shattering} can be combined with Proposition \ref{prop:well_clustered} and Corollary \ref{cor:vectorization_Lloyd} to provide guarantees on the output of Algorithm \ref{lst:lloyd} combined with a suitable kernel. We choose to give results for the theoretical kernel $\psi_0:x \mapsto (1- ((x-1) \vee 0)) \vee 0$, and for the kernel used in \cite{Royer19}, $\psi_{AT} = x \mapsto \exp(-x)$.
\begin{corollary}\label{cor:PD_vectorization_clustering}
Assume that $\E(X)$ satisfies a margin condition, and that the requirements of Proposition \ref{prop:PD_vectorization_shattering} are satisfied. For short, denote by $v_i$ the vectorization of $X_i$ based on the output of Algorithm \ref{lst:lloyd}. Then, with probability larger than $1 - \exp{\left [-C\left ( \frac{n r^2 p_{min}^2}{p^2M^2R^2k^2 d \log(k)} - \frac{p_{min}^2 B^2 r_0^2}{M^2 R^4 k^2 d \log(k)} \right ) \right ]} - n \max_{\ell \leq L}N_\ell^{-\left ( \frac{(\kappa \tilde{B})^{d_\ell} f_{min,\ell} N_\ell}{C_\ell d_\ell \log(N_\ell)}\right)} $, where $\kappa$ and $C$ are small enough constants, we have
\[
\begin{array}{@{}ccc}
Z_{i_1}=Z_{i_2} & \Rightarrow & \| v_{i_1} - v_{i_2} \|_\infty \leq \frac{1}{4}, \\
Z_{i_1} \neq Z_{i_2} & \Rightarrow & \| v_{i_1} - v_{i_2} \|_\infty \geq \frac{1}{2},
\end{array}
\]
for $\sigma \in [r,2r]$ and the following choices of $p$ and $r$:
\begin{itemize}
\item If $\Psi = \Psi_{AT}$, $p_{AT} = \lceil 4M  \rceil$, and $r_{AT} = \frac{\tilde{B}}{32 p_{AT}}$.
\item If $\Psi = \Psi_0$, $p_0=1$ and $r_0 = \frac{\tilde{B}}{32}$.
\end{itemize}
\end{corollary}
A proof of Corollary \ref{cor:PD_vectorization_clustering} is given in Section \ref{sec:proof_cor_PD_vectorization_clustering}. Corollary \ref{cor:PD_vectorization_clustering} can be turned into probability bounds on the exactness of the output of hierarchical clustering schemes applied to the sample points. For instance, on the probability event described by Corollary \ref{cor:PD_vectorization_clustering}, Single Linkage with norm $\|.\|_\infty$ will provide an exact clustering. The probability bound in Corollary \ref{cor:PD_vectorization_clustering} shed some light on the quality of sampling of each shape that is required to achieve a perfect classification: roughly, for $N_\ell$ in $\Omega(\log(n))$, the probability of misclassification can be controlled. Note that though the key parameter $\tilde{B}$ is not known, in practice it can be scaled as several times the minimum distance between two points of a diagram.  

Investigating whether the margin condition required by Corollary \ref{cor:PD_vectorization_clustering} is satisfied in the mixture of shapes framework is beyond the scope of the paper. This condition is needed to prove that the algorithms exposed in Section \ref{sec:Algorithms} provide provably good approximations of optimal $\cb^*$'s (those in turn can be proved to be discriminative codebooks). Experimental results below asses the validity of our algorithms in practice.

\section{Experimental results}\label{sec:Expes}

This section investigates numerical performances of our vectorization and clustering procedure. The vectorization step is carried out using  the \textsc{Atol} procedure (Automatic Topologically-Oriented Learning, \cite{Royer19}): codebooks are defined using Algorithm \ref{lst:macqueen}, and an exponential kernel $\Psi_{AT}$ is used to construct the embedding as in \eqref{eq:vectorization}.  As a remark, Algorithm \ref{lst:lloyd} combined with other kernels such as $\psi_0$ lead to similar results in our three experiments.

It is worth noting that our method requires few to no tuning: for all experiments we use minibatches of size 1000 and the same calibration scheme (use a random 10\% of the sample measures for learning the space quantization). We use the same kernel $\Psi_{AT}$, only the \textsc{Atol} budget $k$ (length of the vectorization) will sometimes vary in the three experiments. 

Our numerical experiments encompass synthetic measure clustering, but also investigate the performance of our vectorization technique on two (supervised) classification problems: large-scale graph classification and text classification. On these problems, we achieve state-of-the-art performances.

\subsection{Measure clustering}

In this first experiment, we settle in the framework of Section \ref{sec:measures_clustering}. Namely, we consider an i.i.d sample drawn from a mixture of measures, and evaluate our method on a clustering task. 
We generate a synthetic mixture of $L$ measures. We first choose $p-1$ centers shared by all mixture components on the $d-1$-dimensional sphere of radius $10$ and denote them by $\{S_d^{i} \mid i=1, \hdots, p-1\}$; then for each mixture component another center is placed at a separate vertex of the $d$-dimensional unit cube (implying $L \leqslant 2^d$), labeled $\text{Cube}_{\ell|d}$ for $\ell\in [\![1, \hdots, L]\!]$, so that the inner center is the only center that varies amongst mixture components, see instances Figure \ref{fig:setup}. The \textit{support centers} $C^\ell$ for the mixture component $\ell$ are given by
\begin{align}
C^\ell = \{ S_d^{i} \mid i=1, \hdots, p-1\} \cup \text{Cube}_{\ell|d}.
\end{align}
For every mixture component and signal level $r>0$, we then make $N$ normal simultaneous draws with variance $1$ centered around every element of  $r \times C^\ell$, resulting in a point cloud of cardinality $p \times N$ that is interpreted as a measure. To sum up, the $\ell$-th mixture component $X^k$ has distribution
\[
X^\ell = \bigcup_{c \in C^\ell} \bigcup_{i=1, \hdots, N} \{ rc + \varepsilon_{c,i} \}, 
\]
where the $\varepsilon$'s are i.i.d standard $d$-dimensional Gaussian random variables. 

We compare with several measure clustering procedures, some of them being based on alternative vectorizations. Namely, the \texttt{rand} procedure randomly chooses codepoints among the sample measures points, whereas the \texttt{grid} procedure takes as codepoints a regular grid of $[0,10r]^d$. For fair comparison we use the same kernel $\Psi_{AT}$ as in \textsc{Atol} to vectorize from these two codebooks.

We also compare with two other standard clustering procedures: the \texttt{histogram} procedure, that learns a regular tiling of the space and vectorizes measures by counting inside the tiles, and the \texttt{W2-spectral} method that is just spectral clustering based on the $W_2$ distance (to build the Gram matrix). Note that the \texttt{histogram} procedure is conceptually close to the \texttt{grid} procedure.

At last, for all methods the final clustering step consists in a standard k-means clustering with 100 initializations.

\begin{figure}[h]
	\centering
	\includegraphics[width=.45\columnwidth]{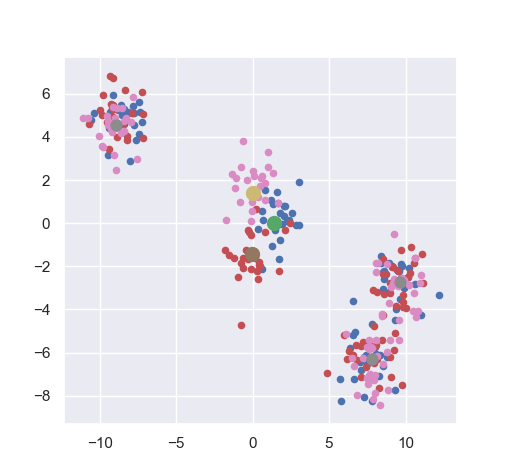}
	\includegraphics[width=.45\columnwidth]{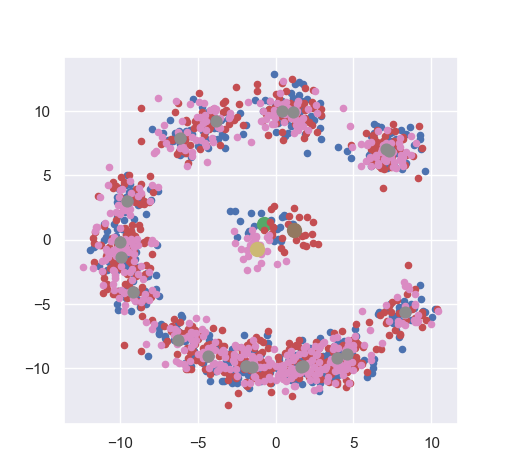}
	\caption{Synthetic 3-mixture of measures instance (blue, dark brown and pink) and their \textit{support centers} for indication (dark green for the outer centers, yellow, green and brown for the middle centers) with either 4 (left) or 20 (right) \textit{support centers} in dimension 2, $r=1$. The sets of points similarly coloured constitute measures.}
	\label{fig:setup}
\end{figure}

We designed these mixtures to exemplify problems where the data are significantly separated (by $L$ discriminative points corresponding to hypercube vertices), but also very similar by including draws from the same set of centers on the outer sphere. In this experiment $r$ may be thought of as the strength of the signal separation, and $p$ can be interpreted as a measure of noise. Empirically we use two different setups for the number of supporting centers $p$: $p=4$ or $p=20$, that we interpret respectively as low and high noise situations in Figure \ref{fig:setup}. We fix $L=3$, and generate measures in dimension $2$ and $5$. We also set $N=25$ and sample every mixture component $20$ times, so that the final sample of measures consist of $60$ i.i.d. finite measures, each with $25\times p$ points.

Regardless of the dimension, for the $p=4$ setup, \texttt{W2-spectral} takes about 4.3 seconds on average, 0.2 seconds for \textsc{Atol}, and the other methods are naturally faster. Note that for the $p=20$ case, the \texttt{W2-spectral} clustering procedure runs takes too long to compute and is not reported, whereas \textsc{Atol} runs in $0.5$ seconds on average. Each experiment is performed identically a hundred times and the resulting average \textit{normalized mutual information} (NMI) with true labels and $95\%$ confidence intervals are presented for each method.
Two sets of experiments are performed, corresponding to the following two questions.

\textbf{Q1}: at a given signal level $r=1$, for increasing budget (that is, the size of the codebooks used in the vectorization process), how accurate are methods at clustering 3-mixtures? Q1 aims at comparing the various methods potentials at capturing the underlying signal with respect to budget.

Results are displayed in Figure \ref{fig:mmc-budget}. The \textsc{Atol} procedure is shown to produce the best results in almost all situations. Namely, it is the procedure that will most likely cluster the mixture exactly while requiring the least amount of budget for doing so. The grid procedures \texttt{grid} and \texttt{histogram} naturally suffer from the dimension growth (bottom panels). The oscillations of the \texttt{histogram} procedure performances, exposed in the upper right panel,  may be due to the existence of a center tile (that cannot discriminate between several vertices of the cube) that our oversimple implementation allows. To draw a fair comparison, the local minima of the \texttt{histogram} curve are not considered as representative, leading to an overall flat behavior of this method in the case $d=2, p=20$. In dimension 2 the \texttt{grid} procedure shows to be very efficient, but for measures in dimension 5 it seems always better to select codepoints at random, rather than using some instance of regular grid. Whenever possible, \texttt{W2-spectral} outperforms all vectorization-based methods, at the price of a significantly larger computational cost. Note however that the performance gap with \textsc{Atol} is discernable for small budgets only ($k \leq 15$).

\begin{figure}[h!]
	\centering
	\includegraphics[width= \columnwidth]{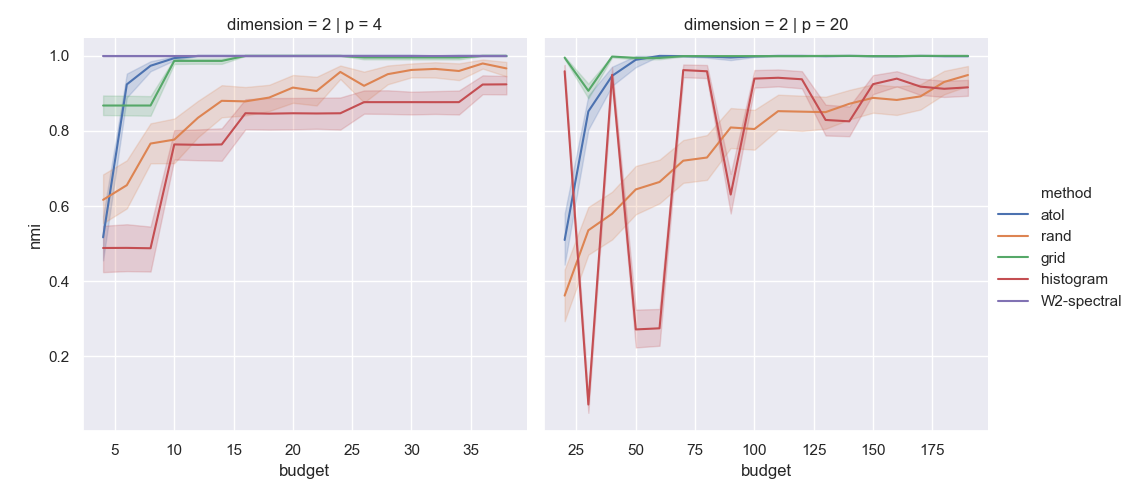}
	\includegraphics[width= \columnwidth]{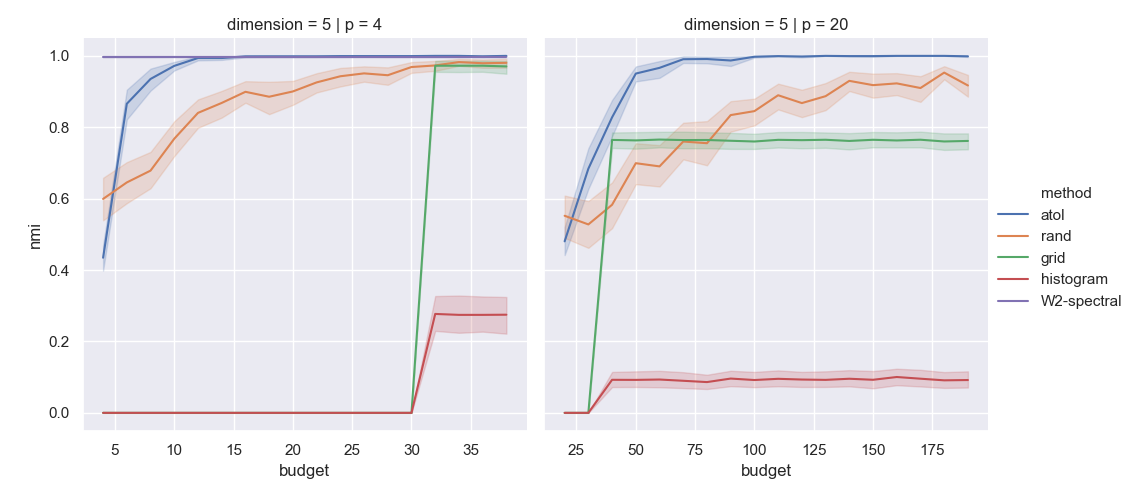}
	\caption{NMI as budget increases, $r=1$, increasing dimensions along rows, low noise and high noise contexts (left and right columns)}
	\label{fig:mmc-budget}
\end{figure}

\textbf{Q2}: for a given budget $k=32$ and increasing signal level, how accurate are methods at clustering 3-mixtures? Q2 investigates how the various methods fare with respect to signal strength at a given, low codebook length $k=32$, the most favorably low budget for regular grid procedures in dimension 5.

Results are shown in Figure \ref{fig:mmc-snr}. \textsc{Atol} and \texttt{grid} procedures give similar results, and compare well to \texttt{W2-spectral} in the case $p=4$. Figure \ref{fig:mmc-snr} also confirms the intuition that the \texttt{rand} procedure is more sensitive to $p$ than competitors, resulting in significantly weaker performances when $p=20$.

\begin{figure}[h!]
	\centering
	\includegraphics[width= \columnwidth]{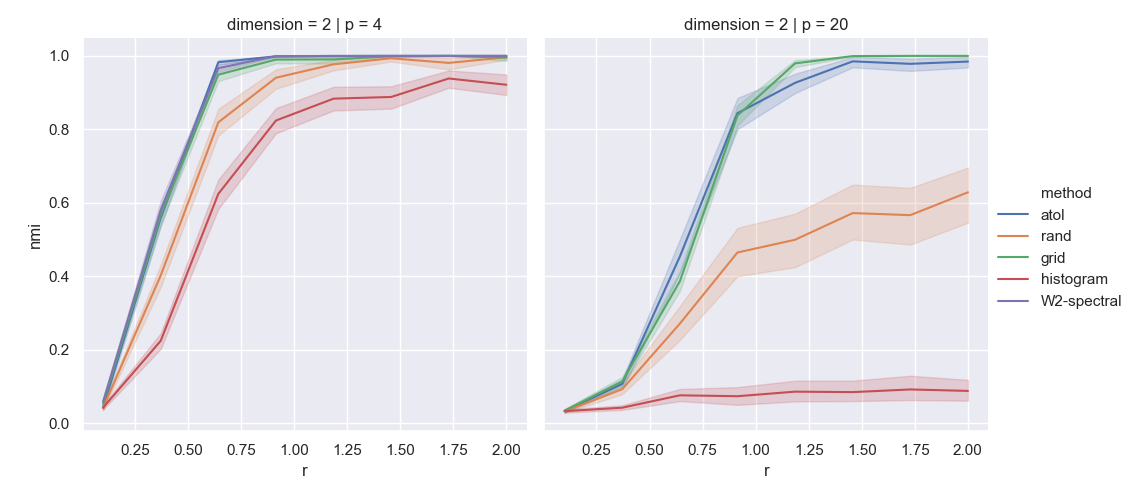}
	\includegraphics[width= \columnwidth]{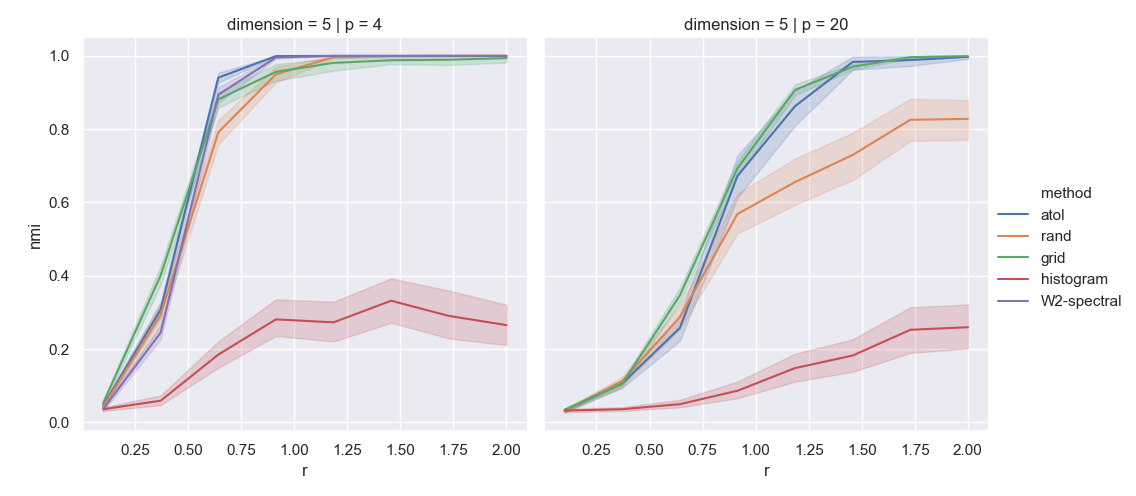}
	\caption{NMI as $r$ increases, budget $k = 32$, increasing dimensions along rows, low noise and high noise contexts (left and right columns)}
	\label{fig:mmc-snr}
\end{figure}

\subsection{Large-scale graph classification}

We now look at the large-scale graph classification problem. Provided that graphs may be considered as measures, our vectorization method provides an embedding into $\R^k$ that may be combined with standard learning algorithms. In this framework, our vectorization procedure can be thought of as a dimensionality reduction technique for supervised learning, that is performed in an unsupervised way. Note that the notion of shattering codebook introduced in Section \ref{sec:measures_clustering} is still relevant in this context, since a $(p,r,\Delta)$-shattering codebook would lead to exact sample classification if combined with a classification tree for instance. Theoretically investigating the generalization error of such an approach is beyond the scope of the paper. 
    
As exposed above, since \textsc{Atol} is an unsupervised procedure that is not specifically conceived for graphs, it is likely that other dedicated techniques would provide better performances in this graph classification problem. However, its simplicity can be a competitive advantage for large-scale applications. 

There are multiple ways to interpret graphs as measures. In this section we borrow from \cite{Carriere19}: for a  diffusion time $t>0$ we compute the Heat Kernel Signatures (HKS$_t$) for all vertices in a graph, so that each graph $G(V, E)$ is embedded in $\mathbb{R}^{|V|}$ (see, e.g., \cite[Section 2.2]{Carriere19}). Then four graph descriptors per diffusion time may be computed using the extended persistence framework (see, e.g.,  \cite[Section 2.1]{Carriere19}), that is four types of persistence diagrams (PDs). Schematically for a graph $G(V, E)$ the descriptors are derived as:
\begin{align}
G(V, E) &\xrightarrow[\text{signatures}]{\text{heat kernel}} \text{HKS}_{t}(G) \in \mathbb{R}^{|V|}, \notag \\
	G(V, E) &\xrightarrow[\text{persistence}]{\text{extended}} \text{PD}(\text{HKS}_{t}(G), G) \in (\mathcal{M}(\mathbb{R}^2))^{4}.	\label{eq:graphpds}
\end{align}

In these experiments we will show that the graph embedding strategy of \eqref{eq:graphpds} paired with \textsc{Atol} can perform up to the state of the art on large-scale graph classification problems. We will also demonstrate that \textsc{Atol} is well assisted but not dependent on the aforementioned TDA tools.

\noindent
\textbf{Large-scale binary classification from \cite{rozemberczki2020api}}

Recently \cite{rozemberczki2020api} introduced large-scale graph datasets of social or web origin. For each dataset the associated task is binary classification. The authors perform a 80\% train/test split of the data and report mean area under the curve (AUC) along with standard errors over a hundred experiments for all the following graph embedding methods. SF from \cite{lara2018simple} is a simple graph embedding method that extracts the $k$ lowest spectral values from the graph Laplacian, and a standard Random Forest Classifier (RFC) for classification. NetLSD from \cite{tsitsulin2018} uses a more refined representation of the graph Laplacian, the heat trace signature (a global variant of the HKS) of a graph using 250 diffusion times, and a 1-layer neural network (NN) for classification. FGSD from \cite{verma2017hunt} computes the biharmonic spectral distance of a graph and uses histogram vectorization with small binwidths that results in vectorization of size $[100, 1000000]$, with a Support Vector Machine (SVM) classifier. GeoScattering from \cite{gao19} uses graph wavelet transform to produce 125 graph embedding features, also with a SVM classifier.

We add our own results for \textsc{Atol} paired with \eqref{eq:graphpds}: we use the extended persistence diagrams as input for \textsc{Atol} computed from the HKS values with diffusion times $t_1=.1, t_2=10$, and vectorize the diagrams with budget $k=10$ for each diagram type and diffusion time so that the resulting vectorization for graph $G$ is $v_{\textsc{Atol}}(G) \in \mathbb{R}^{2 \times 4 \times 10}$. We then train a standard RFC (as in \cite{lara2018simple}, we use the implementation from \texttt{sklearn} \cite{sklearn} with all default parameters) on the resulting vectorized measures.

\begin{table*}
	\resizebox{\columnwidth}{!}{ 
		\begin{tabular}{|lr|rrrr|r|}
			\hline
			\hspace{1.2cm} method && SF & NetLSD & FGSD & GeoScat & \textsc{Atol} \\
			problem & size & \cite{lara2018simple} & \cite{tsitsulin2018} & \cite{verma2017hunt} & \cite{gao19} & \\
			\hline
			$\texttt{reddit threads}$ & (203K) & {81.4}$\pm$.2 & \textbf{82.7$\pm$.1} & 82.5$\pm$.2 & 80.0$\pm$.1 & 80.7$\pm$.1 \\
			$\texttt{twitch egos}$ & (127K) & 67.8$\pm$.3 & {63.1}$\pm$.2 & \textbf{70.5$\pm$.3} & {69.7$\pm$.1} & 69.7$\pm$.1 \\
			$\texttt{github stargazers}$ & (12.7K) & {55.8}$\pm$.1 & 63.2$\pm$.1 & {65.6$\pm$.1} & 54.6$\pm$.3 & \textbf{72.3$\pm$.4} \\
			$\texttt{deezer ego nets}$ & (9.6K) & 50.1$\pm$.1 & 52.2$\pm$.1 & \textbf{52.6$\pm$.1} & 52.2$\pm$.3 & {51.0$\pm$.6} \\
			\hline
		\end{tabular}
	}
	\caption{Large graph binary classification problems. Mean ROC-AUC and standard deviations.}
	\label{tab:snap_graphs}
\end{table*}

The results are shown Table \ref{tab:snap_graphs}. \textsc{Atol} is close to or over the state-of-the-art for all four datasets. Most of these methods operate directly from graph Laplacians so they are fairly comparable, in essence or as for the dimension of the embedding that is used. The most positive results on \texttt{github stargazers} improves on the best method by more than 6 points. 

\noindent
\textbf{A variant of the large-scale graph classification from \cite{Royer19}}

The graph classification tasks above were binary classifications. \cite{Yanardag15} introduced now popular datasets of large-scale graphs associated with multiple classes. These datasets have been tackled with top performant graph methods including the graph kernel methods RetGK from \cite{zhang2018retgk}, WKPI from \cite{qi2019} and GNTK from \cite{du19} (combined with a graph neural network), and the aforementioned graph embedding method FGSD from \cite{verma2017hunt}. PersLay from \cite{Carriere19} is not designed for graphs but used \eqref{eq:graphpds} as input for a 1-NN classifier. Lastly, in \cite{Royer19}, the authors also used \eqref{eq:graphpds} and the same diagrams as input data for the \textsc{Atol} procedure and obtained competitive results within reasonable computation times (the largest is $110$ seconds, for the \texttt{COLLAB} dataset, see \cite[Section 3.1]{Royer19} for all computation times).

Here we propose to bypass the topological feature extraction step in \eqref{eq:graphpds}. Instead of topological descriptors, we use simpler graph descriptors: we compute four HKS descriptors corresponding to diffusion times $t_1=.1, t_2=1, t_3=10, t_4=100$ for all vertices in a graph, but this time directly interpret the output as a measure embedding in dimension 4. From there we use \textsc{Atol} with $k=80$ budget. Therefore each graph $G(V, E)$ is embedded in $\mathbb{R}^{4|V|}$ seen as $\mathcal{M}(\mathbb{R}^4)$ and our measure vectorization framework is readily applicable from there. To sum-up we now use the point of view:

\begin{align}
G(V, E) \xrightarrow[\text{signatures}]{\text{heat kernel}} \text{HKS}_{t_1, t_2, t_3, t_4}(G) \in \mathbb{R}^{4|V|} \approx \mathcal{M}(\mathbb{R}^4).
\label{eq:graphdirect}
\end{align}

It is important to note that the proposed transformation from graphs to measures can be computed using any type of node or edge embedding. Our vectorization method could be applied the same way as for HKS's. 

\begin{table*}[!h]
	\resizebox{\columnwidth}{!}{ 
		\begin{tabular}{|l|rrrrrr|r|}
			\hline
			\hspace{1.2cm} method & RetGK & FGSD & WKPI & GNTK & PersLay & \textsc{Atol} & \textsc{Atol} \\
			problem & \cite{zhang2018retgk} & \cite{verma2017hunt} & \cite{qi2019} & \cite{du19} & \cite{Carriere19} & with \eqref{eq:graphpds} \cite{Royer19} & with \eqref{eq:graphdirect} \\
			\hline
			$\texttt{REDDIT}$ (5K, 5 classes)       & 56.1$\pm$.5 & 47.8 & {59.5$\pm$.6} & --- & {55.6}$\pm$.3 & \textbf{67.1$\pm$.3} & \textbf{66.1$\pm$.2} \\
			$\texttt{REDDIT}$ (12K, 11 classes)      & {48.7}$\pm$.2 & ---  & {48.5$\pm$.5} & --- & {47.7}$\pm$.2 & \textbf{51.4$\pm$.2} & \textbf{50.7$\pm$.3} \\
			$\texttt{COLLAB}$ (5K, 3 classes)         & {81.0}$\pm$.3 & 80.0 & --- & 83.6$\pm$.1 & 76.4$\pm$.4 & \textbf{88.3$\pm$.2} & \textbf{88.5$\pm$.1} \\
			$\texttt{IMDB-B}$ (1K, 2 classes)        & 71.9$\pm$1. & {73.6} & 75.1$\pm$1.1 & \textbf{76.9$\pm$3.6} & 71.2$\pm$.7 & 74.8$\pm$.3 & 73.9$\pm$.5 \\
			$\texttt{IMDB-M}$ (1.5K, 3 classes)      & 47.7$\pm$.3 & \textbf{52.4} & 48.4$\pm$.5 & \textbf{52.8$\pm$4.6} & 48.8$\pm$.6 & 47.8$\pm$.7 & 47.0$\pm$.5 \\
			\hline
		\end{tabular}
	}
	\caption{Mean accuracies and standard deviations for the large multi-classes problems of \cite{Yanardag15}.}
	\label{tab:res_social}
\end{table*}

On Table \ref{tab:res_social} we quote results and competitors from \cite{Royer19} and on the right column we add our own experiment with \textsc{Atol}. The \textsc{Atol} methodology works very efficiently with the direct HKS embedding as well, although the results tend to be slightly inferior. This may hint at the fact that although PDs are not essential to capturing signal from this dataset, they can be a significant addition for doing so. 
Overall \textsc{Atol} with \eqref{eq:graphdirect}, though much lighter than \textsc{Atol} with \eqref{eq:graphpds}, remains competitive for large-scale graph multiclass classification.


\subsection{Text classification with word embedding}

At last we intend to apply our methodology in a high-dimensional framework, namely text classification.
Basically, texts are sequences of words and if one forgets about word order, a text can be thought of as a measure in the (fairly unstructured) space of words. We use the \texttt{word2vec} word embedding technique (\cite{Mikolov13}), that uses a two-layer neural network to learn a real-valued vector representation for each word in a dictionary in such a way that distances between words are learnt to reflect the semantic closeness between them, with respect to the corpus. Therefore, for a given dimension $E \in \mathbb{N}$, every word $w$ is mapped to the embedding space $v_{\texttt{word2vec}}(w) \in \mathbb{R}^E$, and we can use word embedding to interpret texts as measures in a given word embedding space:
\begin{align}
T \in \mathcal{M}(\{\text{words}\}) \xrightarrow[\text{embedding}]{\text{word}} T_{\texttt{word2vec}} = \big[ v_{\texttt{word2vec}}(w) \big]_{w \in T} \in \mathcal{M}(\mathbb{R}^{E}).
\label{eq:embedding}
\end{align}

In practice we will use the Large Movie Review Dataset from \cite{imdb} for our text corpus, and learn word embedding on this corpus. To ease the intuition, Figure \ref{fig:imdbquant} below depicts the codepoints provided by Algorithm \ref{lst:macqueen}, based on a $2$-dimensional word embedding. Note that this word embedding step does not depend on the classification task that will follow.

\begin{figure}[h!]
	\centering
	\includegraphics[width= \columnwidth]{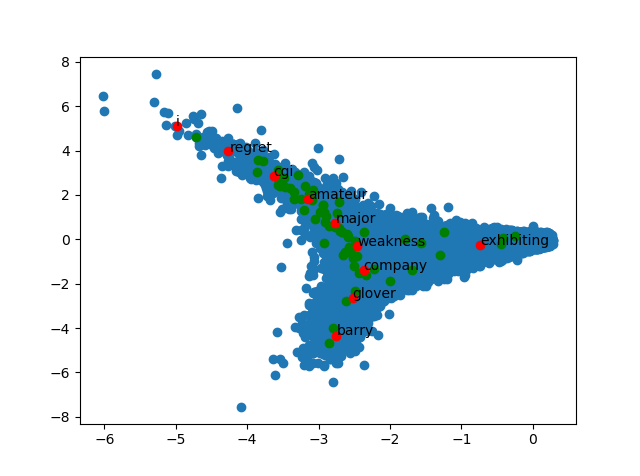}
	\caption{Blue: point cloud embedding in dimension 2 for the words of all reviews; green: point cloud embedding for the words of the first review ("Story of a man who has unnatural feelings for a pig. [...]"); red: codepoints computed by Algorithm \ref{lst:macqueen}, with the text print of their closest corresponding word.}
	\label{fig:imdbquant}
\end{figure}

We now turn to the task of classification and use the binary sentiment annotation (positive or negative reviews) from the dataset. For classifying texts with modern machine-learning techniques there are several problems to overcome: important words in a text can be found at almost any place, texts usually have different sizes so that they have to be padded in some way, etc. The measure viewpoint is a direct and simple solution.

The successful way to proceed after a word embedding step is to use some form of deep neural network learning. We learn a $100$-dimensional word embedding using the \texttt{word2vec} implementation from the \texttt{gensim} module, then compare the following classification methods: directly run a recurrent neural network (LSTM with 64 units), against the \textsc{Atol} vectorization followed by a simple dense layer with 32 units. We measure accuracies through a single 10-fold.

The \textsc{Atol} vectorization with word embedding in dimension 100, $20$ centers, and 1-NN classifier reaches 85.6 accuracy with .95 standard deviation. The average quantization and vectorization times are respectively 5.5 and 208.3 seconds. The recurrent neural network alternative with word embedding in dimension 100 reaches 89.3 mean accuracy with .44 standard deviation, for an average run time of about $1$ hour. 

Naturally the results are overall greatly hindered by the fact that we have treated texts as a separate collection of words, forgetting sentence structure when most competitive methods use complex neural networks to analyse $n$-uplets of words as additional inputs. Additional precision can be also gained using a task-dependend embedding of words, at the price of a larger computation time (see for instance the \texttt{kaggle} winner algorithm \cite{Cherniuk19}, with precision $99\%$ and run time $10379$ seconds).

\section{Proofs for Section \ref{sec:measures_clustering_vectorization}}
\label{sec:proof_sec_clustering_based_on_vectorization}
\subsection{Proof of Proposition \ref{prop:vectorization_scattering}}\label{sec:proof_prop_vectorization_scattering}  
Assume that $X_1, \hdots, X_n$ is $(p,r,\Delta)$-shattered by $\cb$, let $i_1, i_2$ in $[\![1,n]\!]$ be such that $Z_{i_1} \neq Z_{i_2}$, and without loss of generality assume that
\[
X_{i_1}(\B(c_1,r/p)) \geq  X_{i_2}(\B(c_1,4rp)) + \Delta.
\]
Let $\Psi$ be a $(p,\delta)$-kernel and $\sigma \in [r,2r]$. We have
\begin{align*}
\int \Psi(\|u-c_1\|/\sigma) X_{i_1}(du) & \geq  \int\Psi(\|u-c_1\|/\sigma) \mathds{1}_{\B(c_1,r/p)}(u) X_{i_1}(du) \\
& \geq (1-\delta) X_{i_1}\left(\B(c_1,r/p)\right )\\
& \geq X_{i_1}\left(\B(c_1,r/p) \right ) - \delta M. 
\end{align*}
On the other hand, we have that
\begin{align*}
 \int \Psi(\|u-c_1\|/\sigma)X_{i_2}(du) \leq & \begin{multlined}[t]  \int \Psi(\|u-c_1\|/\sigma) \mathds{1}_{\B(c_1,4pr)}  X_{i_2}(du) \\ +  \int \Psi(\|u-c_1\|/\sigma) \mathds{1}_{(\B(c_1,4pr))^c} X_{i_2}(du)
\end{multlined} \\
& \leq X_{i_2}(\B(c_1,4pr)) + \delta X_{i_2}((\B(c_1,4pr))^c) \\
& \leq X_{i_1}\left(\B(c_1,r/p) \right ) - \Delta + \delta M.
\end{align*}
We deduce that $\|v_{\cb,\sigma}(X_{i_1}) -v_{\cb,\sigma}(X_{i_2}) \|_\infty \geq \Delta - 2 \delta M \geq \frac{\Delta}{2}$
whenever $\delta \leq \frac{\Delta}{4M}$.

\subsection{Proof of Proposition \ref{prop:well_clustered}}\label{sec:proof_prop_well_clustered}
Let $i_1$, $i_2$ in $[\![1,n]\!]$ such that $Z_{i_1} = Z_{i_2}$. Let $(Y_1,Y_2)$ be a random vector such that $Y_1 \sim X_{i_1}$, $Y_2 \sim X_{i_2}$, and $\mathbb{E}(\|Y_1-Y_2\|) \leq w$. Let $c \in \B(0,R)$, we have
\begin{multline*}
\left |  \int \Psi(\|u-c\|/\sigma) X_{i_1}(du) -  \int \Psi(\|u-c\|/\sigma) X_{i_2}(du) \right | \\
\begin{aligned}[t]
& \leq \left | \mathbb{E}\left [ \Psi(\|Y_1 - c\|/\sigma) - \Psi(\|Y_2 - c\|/\sigma) \right ] \right | \\
& \leq \E \left ( \frac{\|Y_1-Y_2\|}{\sigma} \right ) \leq \frac{w}{\sigma}, 
\end{aligned}
\end{multline*} 
hence $\|v_{\cb,\sigma}(X_{i_1})-v_{\cb,\sigma}(X_{i_2})\|_\infty \leq w/\sigma$. Now if $X_1 \hdots, X_n$ is $r \Delta/4$-concentrated, and $\sigma \in [r,2r]$, we have $\|v_{\cb,\sigma}(X_{i_1})-v_{\cb,\sigma}(X_{i_2})\|_\infty \leq \frac{\Delta}{4}$.

\section{Proofs for Section \ref{sec:mean_measure_quantization}}\label{sec:proof_sec_mean_measure_quantization}
\subsection{Proof of Theorem \ref{thm:optLloyd}}\label{sec:proof_thm_optLloyd}
Throughout this section we assume that $\mathbb{E}(X)$ satisfies a margin condition with radius $r_0$, and that $\mathbb{P}(X \in \mathcal{M}_{N_{max}}(R,M)) =1$. 
The proof of Theorem \ref{thm:optLloyd} is based on the following lemma, which ensures that every step of Algorithm \ref{lst:lloyd} is, up to concentration terms, a contraction towards an optimal codebook. We recall here that $R_0= \frac{Br_0}{16\sqrt{2}R}$, $\kappa_0 = \frac{R_0}{R}$.
\begin{lem}\label{lem:Lloyd_one_iteration}
Assume that $\cb^{(0)} \in \B(\cb^*,R_0)$. Then, with probability larger than $1-9e^{ -c_1 n p_{min}/M} - e^{-x}$, for $n$ large enough, we have, for every $t$,
\begin{align}\label{eq:lloyd_distorsion}
\|\cb^{(t+1)} - \cb^*\|^2 \leq \frac{3}{4} \|\cb^{(t)} - \cb^*\|^2 + \frac{K}{p_{min}^2} D_n^2,
\end{align}
where $D_n = \frac{C R M}{\sqrt{n}}  \left (k \sqrt{d\log(k)}+ \sqrt{x} \right )$ and $C,K$ are positive constants.
\end{lem}
The proof of Lemma \ref{lem:Lloyd_one_iteration} is deferred to \Cref{sec:proof_lemma_Lloyd_one_iteration}.
\begin{proof}[Proof of Theorem \ref{thm:optLloyd}]
Equipped with Lemma \ref{lem:Lloyd_one_iteration}, the proof of Theorem \ref{thm:optLloyd} is straightforward.  On the probability event of Lemma \ref{lem:Lloyd_one_iteration}, using \eqref{eq:lloyd_distorsion} we have that 
\begin{align*}
\| \cb^{(t)} - \cb^* \|^2 & \leq \left ( \frac{3}{4} \right )^{t}\| \cb^{(0)} - \cb^*\|^2 + \left ( \sum_{p=0}^{t-1} \left ( \frac{3}{4} \right )^p \right ) \frac{K}{p_{min}^2} D_n^2 \\
& \leq \left ( \frac{3}{4} \right )^{t}\| \cb^{(0)} - \cb^*\|^2 + \frac{4K}{p_{min}^2}D_n^2.
\end{align*}
\end{proof}

\subsection{Proof of Theorem \ref{thm:opt_Mc_Queen}}\label{sec:proof_thm_opt_Mc_Queen}
The proof of Theorem \ref{thm:opt_Mc_Queen} follows the proof of \cite[Lemma 1]{Rakhlin11}. 
Throughout this section we assume that $\E(X)$ satisfies a margin condition with radius $r_0$. The proof of Theorem \ref{thm:opt_Mc_Queen} is based on the following lemma. Recall that $R_0= \frac{Br_0}{16\sqrt{2}R}$, $\kappa_0 = \frac{R_0}{R}$. We let $\cb^{(t)}$ denote the output of Algorithm \ref{lst:macqueen} (or its variant in the sample point case).
\begin{lem}\label{lem:onestep_McQueen}
Assume that $\cb^{(0)} \in \mathcal{B}(\cb^*,R_0)$, and $n_t \geq c_0 \frac{k^2M^2}{p_{min}^2 \kappa_0^2} \log(n)$, for some constant $c_0$, with $n \geq k$. Then we have, for any $t =0, \hdots, T-1$,
\begin{align}\label{eq:mcqueen_iter_expectation}
\E \left ( \| \cb^{(t+1)} - \cb^*\|^2 \right ) \leq \left ( 1 - \frac{2-K_1}{t+1} \right ) \E \left ( \| \cb^{(t)} - \cb^* \|^2 \right ) + \frac{16kMR^2}{p_{min}(t+1)^2},
\end{align}
with $K_1 \leq 0.5$.
\end{lem}
The proof of Lemma \ref{lem:onestep_McQueen} is deferred to \Cref{tecsec:proof_lemma_onestep_McQueen}.
\begin{proof}[Proof of Theorem \ref{thm:opt_Mc_Queen}]
Equipped with Lemma \ref{lem:onestep_McQueen}, we can prove Theorem \ref{thm:opt_Mc_Queen} using the same method as in the proof of \cite[Lemma 1]{Rakhlin11}. Namely, denoting by $a_t = \| \cb^{(t)} - \cb^*\|^2$, we prove recursively that 
\begin{align*}
\E a_t \leq \frac{32kMR^2}{p_{min}t}.
\end{align*}
Denote by $G=\frac{16kMR^2}{p_{min}}$. The case $t=1$ is obvious. Next, assuming that $\E a_t \leq \frac{2G}{t}$ and using \eqref{eq:mcqueen_iter_expectation} we may write
\begin{align*}
\E a_{t+1} & \leq \left ( 1 - \frac{2}{t+1} \right )\E a_t + \frac{K_1}{t+1}\E a_t + \frac{G}{(t+1)^2} \\
& \leq \frac{G}{t(t+1)} \left [ 2 t +2K_1-1 \right ].
\end{align*}
Since $K_1 \leq \frac{1}{2}$, we get that $\E a_{t+1} \leq 2G/(t+1)$.
\end{proof} 

\subsection{Proof of Corollary \ref{cor:vectorization_Lloyd}}
\label{sec:proof_cor_vectorization_Lloyd}
For $n$ large enough, with probability larger than $1-\exp{\left [-C\left ( \frac{n r^2 p_{min}^2}{p^2M^2R^2k^2 d \log(k)} - \frac{p_{min}^2 B^2 r_0^2}{M^2 R^4 k^2 d \log(k)} \right ) \right ]}$, we have $\| \hat{\cb}_n - \cb^* \| \leq \frac{r}{p}$, according to Theorem \ref{thm:optLloyd}. Let $i_1$, $i_2$ $\in$ $[\![1,n]\!]$ be such that $Z_{i_1} \neq Z_{i_2}$. Without loss of generality assume that 
\[
X_{i_1} \left ( \B(c_1^*,r/p) \right ) \geq  X_{i_2} \left ( \B(c_1^*,4pr) \right ) + \Delta.
\]
Then $X_{i_1} \left ( \B(\hat{c}_1,2r/p) \right ) \geq X_{i_1} \left ( \B(\hat{c}_1,r/p)) \right )$, combined with $X_{i_2} \left ( \B(\hat{c}_1,4(p/2)r) \right ) \leq X_{i_2} \left( \B(c_1^*,4pr) \right )$ entails that
\begin{align*}
X_{i_1} \left ( \B(\hat{c}_1,2r/p) \right ) \geq X_{i_2} \left ( \B(\hat{c}_1,4(p/2)r) \right ) + \Delta.
\end{align*} 
\subsection{Proof of Proposition \ref{prop:mean_meas_closeto_disc_points}}\label{sec:proof_prop_meas_closeto_disc_points}

The proof of Proposition \ref{prop:mean_meas_closeto_disc_points} relies on the two following lemmas. First, Lemma \ref{lem:pers_diag_stability} below ensures that the persistence diagram transformation is stable with respect to the sup-norm. 
\begin{lem}{\cite{CohenSteiner07}}\label{lem:pers_diag_stability}
If $X$ and $Y$ are compact sets of $\R^D$, then 
\[
d_B \left ( D(d_X),D(d_Y) \right ) \leq \| d_X - d_Y \|_{\infty}.
\]
\end{lem}
To apply Lemma \ref{lem:pers_diag_stability} to $d_{\mathbb{Y}_{N_\ell}}$ and $d_{S}^{(\ell)}$, for $\ell \in [\![1,L]\!]$, a bound on $\|d_{\mathbb{Y}_{N_\ell}}-d_{S}^{(\ell)}\|_\infty$ may be derived from the following result.  
\begin{lem}{\cite[Lemma B.7]{Aamari19}}\label{lem:abstandard}
Let $S \subset \R^D$ be a $d$-dimensional submanifold with positive reach $\tau_S$, and let $\mathbb{Y}_N=Y_1, \hdots, Y_N$ be an i.i.d. sample drawn from a distribution that has a density $f(x)$ with respect to the Hausdorff measure of $S$. Assume that for all $x \in S$, $0 < f_{min} \leq f(x)$, and let $h= \left ( \frac{C_d k \log(N)}{f_{min}N} \right )^\frac{1}{d}$, where $C_d$ is a constant depending on $d$. If $h \leq \tau_S/4$, then, with probability larger than $1- \left ( \frac{1}{N} \right )^{\frac{k}{d}}$, we have
\[
\| d_S - d_{\mathbb{Y}_N} \|_\infty \leq h.
\]
\end{lem}

We let $Z$ denote the latent label variable, so that $X\mid \{Z=\ell\} \sim X^{(\ell)}$, or equivalently $X = X^{(Z)}$. Now, for $\ell \in [\![1,L]\!]$, we let $\alpha_\ell = \left ( \frac{1}{N_\ell} \right )^{d_\ell + \frac{2}{d_\ell}}$, $h_\ell = \left ( \frac{C_{d_\ell} \log(N_\ell)}{f_{min,\ell} N_\ell} \right)^{1/d_\ell}$, for some constants $C_{d_\ell}$, and $A= \{ \|d_{Y_{N_{Z}}} - d_{S_Z}\|_\infty > h_Z \}$, so that $\mathbb{P}(A \mid Z= \ell ) \leq \alpha_\ell$. Also, let $m_{k_0+1}, \hdots, m_{k_0 + K_0(h)}$ be such that $\bigcup_{ \ell =1}^{L} D_{\geq s}^{(\ell)} \setminus \{m_1, \hdots, m_{k_0} \}  \subset \bigcup_{s=1}^{K_0(h)} \B_{\infty}(m_{k_0+s},h)$, and $\mathbf{m} = (m_1, \hdots, m_{k_0 + K_0(h)})$. At last, we let $R = \max_{\ell \leq L} diam(S_\ell)$. 
Choosing the $N_\ell$'s large enough so that $D^{(\ell)}_{\geq s-h_\ell} = D^{(\ell)}_{\geq s}$ and $s/2 > h_\ell$, $X^{(Z)}$ and $D^{(Z)}_{\geq s}$ have the same number of (multiple) points, when $A^c$ occurs. Further, we have
\begin{align*}
F(\mathbf{m}) &= \E \left ( \sum_{\ell = 1} \mathds{1}_{Z=\ell}  \int \min_{j = 1, \hdots, k_0 + K_0(h)} \| u - m_j \|^2 X^{(\ell)}(du) \right ) \\
         & = \begin{multlined}[t] \E \left ( \sum_{\ell = 1} \mathds{1}_{Z=\ell \cap A}  \int \min_{j = 1, \hdots, k_0 + K_0(h)} \| u - m_j \|^2 X^{(\ell)}(du) \right ) \\
          +  \E \left ( \sum_{\ell = 1} \mathds{1}_{Z=\ell \cap A^c}  \int \min_{j = 1, \hdots, k_0 + K_0(h)} \| u - m_j \|^2 X^{(\ell)}(du) \right ),
          \end{multlined} 
\end{align*}
so that
\begin{align*}          
       F(\mathbf{m})  & \leq \E \left ( \sum_{\ell = 1} \mathds{1}_{Z=\ell \cap A} 4R^2 N^{d_\ell} \right ) + \E \left ( \sum_{\ell = 1} \mathds{1}_{Z=\ell \cap A^c} M^{(\ell)} 2 h_\ell^2 \right ) \\        
          & \leq 2h^2 \bar{M} + 4 R^2 \sum_{\ell =1}^L \pi_\ell \alpha_\ell N^{d_\ell}.
\end{align*}
For $N_\ell$ large enough so that $\alpha_\ell N_\ell^{d_\ell} \leq \frac{\bar{M} h_\ell^2}{R^2}$, we have $
F(\mathbf{m}) \leq 6 h^2 \bar{M}$.

On the other hand, let $\cb$ be a $k$-points codebook such that, for every $p \in [\![1,k]\!]$, $\|m_1-c_p\|_\infty > 5 \sqrt{\frac{\bar{M}}{\pi_{min}}}h$. Then we have
\begin{align*}
F(\cb) & \geq \E \left ( \sum_{\ell = 1}^L \mathds{1}_{A^c \cap Z=\ell} X^{(\ell)}(\B_\infty(m_1,h)) \left ( 5 \sqrt{\frac{\bar{M}}{\pi_{min}}} -1 \right )^2 \right ) h^2 \\
& \geq  \E \left ( \sum_{\ell = 1}^L \mathds{1}_{A^c \cap Z=\ell} n^{(\ell)}(m_1) \left ( 5 \sqrt{\frac{\bar{M}}{\pi_{min}}} -1 \right )^2  \right ) h^2.
\end{align*} 
Now let $\ell_0$ be such that $n^{(\ell_0)}(m_1) \geq 1$, and assume that the $N_\ell$'s are large enough so that $\alpha_\ell \leq \frac{1}{2}$. It holds
\begin{align*}
F(\cb) & \geq  \left ( 5 \sqrt{\frac{\bar{M}}{\pi_{min}}} -1 \right )^2 h^2 \pi_{\ell_0}(1- \alpha_{\ell_0}) 
          \geq 8 \bar{M} h^2 > F(\mathbf{m}),
\end{align*}
hence the result.
\subsection{Proof of Proposition \ref{prop:PD_vectorization_shattering}}\label{sec:proof_prop_PD_vectorization_shattering}
We let $M=\max_{\ell \leq L} M_\ell$, $h_0$ be such that $\bar{D}_{\geq s-h_0} = \bar{D}_{\geq s}$.
Let $\kappa \leq \frac{1}{16} \wedge \frac{h_0}{2\tilde{B}}$. Under the assumptions of Proposition \ref{prop:mean_meas_closeto_disc_points}, we choose $N_\ell$, $\ell \leq L$ large enough so that ${5\sqrt{\bar{M}}h}/(\sqrt{\pi_{min}}) \leq \kappa \tilde{B}$.
Next, denote by $\alpha_\ell = N_\ell^{-\left ( \frac{(\kappa \tilde{B})^{d_\ell} f_{min,\ell} N_\ell}{C_\ell d_\ell \log(N_\ell)}\right)}$. Then we have
\begin{align*}
\mathbb{P} \left ( \exists i \in [\![1,n]\!]\mid  d_B(X_i,D^{Z_i}_{\geq s}) > \kappa \tilde{B} \right ) &\leq \sum_{i=1}^n \mathbb{P} \left ( d_B(X_i,D^{Z_i}) > \kappa \tilde{B} \right ) \\
& \leq \sum_{i=1}^n \sum_{\ell =1}^L \pi_\ell \alpha_\ell  \leq n \max_{\ell \leq L} \alpha_\ell.
\end{align*}
For the remaining of the proof we assume that, for $i=1, \hdots, n$, $d_B(X_i,D^{Z_i}_{\geq s}) \leq \kappa \tilde{B}$, that occurs with probability larger than $1- n \max_{\ell \leq L} \alpha_\ell$. Recall that on this probability event, under the assumptions of Proposition \ref{prop:mean_meas_closeto_disc_points},  $X_i$ and $D^{Z_i}_{\geq s}$ have the same number of (multiple) points.  Let $i_1 \neq i_2$, and assume that $Z_{i_1} = Z_{i_2} = z$. Then $W_\infty(X_{i_1},X_{i_2}) = d_B(X_{i_1},X_{i_2}) \leq 2 \kappa \tilde{B}$. Hence $W_1(X_{i_1},X_{i_2}) \leq 2M\kappa \tilde{B}$.

Now assume that $Z_{i_1} \neq Z_{i_2}$, and without loss of generality $m_{Z_{i_1},Z_{i_2}} = m_1$ with $D^{Z_{i_1}}_{\geq s}(\{m_1\}) \geq D^{Z_{i_2}}_{\geq s}(\{m_1\})+1$. Let $(p,r)$ in $\mathbb{N}^* \times \R^+$ be such that $r/p \geq 2 \kappa \tilde{B}$ and $4rp \leq \left( \frac{1}{2} - \kappa \right ) \tilde{B}$. Since $\|c^*_1 - m_1\|_\infty \leq \kappa \tilde{B}$ and $d_B(X_{i_1},D^{Z_{i_1}}_{\geq s}) \leq \kappa \tilde{B} < h_0$, we get 
\[
X_{i_1}\left ( \B(c_1^*,\frac{r}{p}) \right ) = D^{Z_{i_1}}_{\geq s}(\{m_1\}).
\]
On the other hand, since $4rp \leq (\frac{1}{2}-\kappa) \tilde{B}$, we also have $X_{i_2}\left ( \B(c_1^*,4rp) \right ) = D^{Z_{i_2}}_{\geq s}(\{m_1\})$.
Thus $X_1, \hdots, X_n$ is $(p,r,1)$-shattered by $\cb^*$.
\subsection{Proof of Corollary \ref{cor:PD_vectorization_clustering}}\label{sec:proof_cor_PD_vectorization_clustering}
In the case where $\Psi = \Psi_{AT}$, we have that $\Psi_{AT}$ is a $(p,1/p)$ kernel. The requirement $1/p \leq \frac{1}{4M}$ of Proposition \ref{prop:vectorization_scattering} is thus satisfied for $p_{AT} = \lceil 4M \rceil$. On the other hand, choosing $r_{AT} = \frac{\tilde{B}}{32 p_{AT}}$ ensures that $8 r_{AT}p_{AT} \leq (1/2 - \kappa) \tilde{B}$ and $\frac{r_{AT}}{2p_{AT}} \geq 2 \kappa \tilde{B}$, for $\kappa$ small enough. Thus, the requirements of Proposition \ref{prop:PD_vectorization_shattering} are satisfied: $\cb^*$ is a $(2p_{AT},r,1)$ shattering of $X_1, \hdots, X_n$. At last, using Corollary \ref{cor:vectorization_Lloyd}, we have that $\hat{\cb}_n$ is a $(p_{AT},r_{AT},1)$ shattering of $X_1, \hdots, X_n$, on the probability event described by Corollary \ref{cor:vectorization_Lloyd}. It remains to note that $2 \kappa \tilde{B} \leq \frac{r_{AT}}{4}$ for $\kappa$ small enough to conclude that $X_1, \hdots, X_n$ is $\frac{r_{AT}}{4}$-concentrated on the probability event described in Proposition \ref{prop:PD_vectorization_shattering}. Thus Proposition \ref{prop:well_clustered} applies.

The case $\Psi = \Psi_0$ is simpler. Since $\Psi_0$ is a $(1,0)$-kernel, we obviously have that $0 \leq \frac{1}{2M}$, so that the requirement of Proposition \ref{prop:PD_vectorization_shattering} is satisfied. With $p_0 =1$ and $r_0 = \frac{\tilde{B}}{16}$ we immediately get that $r_0/(2p_0) \geq 2 \kappa \tilde{B}$ and $8r_0p_0 \leq (1/2 - \kappa \tilde{B})$, for $\kappa$ small enough, so that $\hat{\cb}$ is a $(p_0,r_0,1)$ shattering of $X_1, \hdots, X_n$. As well, $2M \kappa \tilde{B} \leq \frac{r_0}{4}$, for $\kappa$ small enough. Thus Proposition \ref{prop:well_clustered} applies. 

\bibliography{biblio}

\appendix 
\section{Technical proofs}
\label{sec:proof_tecsec}
To ease understanding the statements of the results are recalled before their proofs.
\subsection{Proofs for Section \ref{sec:proof_sec_mean_measure_quantization}}
A key ingredient of the proofs of Lemma \ref{lem:Lloyd_one_iteration} and \ref{lem:onestep_McQueen} is the following Lemma \ref{lem:gradient_margin}, ensuring that around optimal codebooks the expected gradients of Algorithms 1 and 2 are almost Lipschitz. 
\begin{lem}\label{lem:gradient_margin}
Assume that $\E(X) \in \mathcal{M}(R,M)$ satisfies a margin condition with radius $r_0$, and denote by $R_0 = \frac{Br_0}{16\sqrt{2}R}$. Let $\cb^* \in \mathcal{C}_{opt}$, and $\cb$ such that $\|\cb-\cb^*\| \leq R_0$. Then
\begin{itemize}
\item $\sum_{j=1}^{k}|p_j(\cb) - p_j(\cb^*)| \leq \frac{p_{min}}{64}$,
\item $\sum_{j=1}^{k} \| \int(u-c_j)\1_{W_j(\cb)}(u)\E(X)(du) - p_j(\cb^*) (c_j^*-c_j)\| \leq \frac{p_{min}}{8\sqrt{2}}\|\cb - \cb^*\|$.
\end{itemize}
\end{lem}
The proof of Lemma \ref{lem:gradient_margin} follows from \cite[Section A.3]{AppendixLevrard15}.
\subsubsection{Proof of Lemma \ref{lem:Lloyd_one_iteration}}
\label{sec:proof_lemma_Lloyd_one_iteration}
\begin{lem*}[\ref{lem:Lloyd_one_iteration}]
Assume that $\cb^{(0)} \in \B(\cb^*,R_0)$. Then, with probability larger than $1-9e^{ -c_1 n p_{min}/M} - e^{-x}$, for $n$ large enough, we have, for every $t$,
\begin{align*}
\|\cb^{(t+1)} - \cb^*\|^2 \leq \frac{3}{4} \|\cb^{(t)} - \cb^*\|^2 + \frac{K}{p_{min}^2} D_n^2, \qquad \mbox{\eqref{eq:lloyd_distorsion}} 
\end{align*}
where $D_n = \frac{C R M}{\sqrt{n}}  \left (k \sqrt{d\log(k)}+ \sqrt{x} \right )$ and $C,K$ are positive constants.
\end{lem*}
We adopt the following notation: for any $\cb \in \B(0,R)^k$, we denote by $\hat{p}_j(\cb) = \bar{X}_n(W_j(\cb))$, as well as $p_j(\cb) = \E(X)(W_j(\cb))$. Moreover, we denote by $\hat{m}(\cb)$ (resp. $m(\cb))$ the codebooks satisfying, for $j \in [\![1, k]\!]$, 
\begin{align*}
\hat{m}(\cb)_j  = \frac{ \int  u \1_{W_j(\cb)}(u)\bar{X}_n(du)}{\hat{p}_j(\cb)}, \quad 
m(\cb)_j  = \frac{ \int u \1_{W_j(\cb)}(u) \E(X)(du)}{p_j(\cb)},
\end{align*}
if $\hat{p}_j(\cb)>0$ (resp. $p_j(\cb) >0$), and $\hat{m}(\cb)_j = 0$ (resp. $m(\cb)_j=0$) if $\hat{p}_j(\cb)=0$ (resp. $p_j(\cb)=0$). The proof of Lemma \ref{lem:Lloyd_one_iteration} will make use of the following concentration lemma.
\begin{lem}\label{lem:concentration_Lloyd}
With probability larger than $1-8e^{-x}$, for all $\cb \in \B(0,R)^k$, 
\begin{align*}
\hat{p}_j(\cb) & \leq p_j(\cb) + \sqrt{\frac{4Mc_0kd \log(k)  \log( 2 n N_{max})}{n} + \frac{4Mx}{n}}\sqrt{p_j(\cb)} \\
\hat{p}_j(\cb) & \begin{multlined}[t]\geq p_j(\cb) - \frac{4M c_0 kd \log(k) \log({2nN_{max})}}{n} - \frac{4Mx}{n} \\ - \sqrt{\frac{4Mc_0kd \log(k) \log({2nN_{max}})}{n} + \frac{4Mx}{n}}\sqrt{p_j(\cb)},
\end{multlined} 
\end{align*}
where $c_0$ is an absolute constant. Moreover, with probability larger than $1-e^{-x}$, we have
\begin{multline*}
\sup_{\cb \in \B(0,R)^k} \left \| \left (    \int   (c_j-u)\1_{W_j(\cb)}(u) (\bar{X}_n - \mathbb{E}(X))(du)  \right )_{j=1, \hdots, k}  \right \| \\ 
\leq \frac{C R M}{\sqrt{n}}  \left (k \sqrt{d\log(k)}+ \sqrt{x} \right ), 
\end{multline*}
where $C$ is a constant.
\end{lem}
The proof of Lemma \ref{lem:concentration_Lloyd} is given in \Cref{sec:proof_lem_concentration_Lloyd}, and is based on empirical processes theory.
%
\begin{proof}[Proof of Lemma \ref{lem:Lloyd_one_iteration}]
Let $\cb \in \B(\cb^*,R_0)$, and 
\[
D_n = \sup_{\cb \in \B(0,R)^k} \left \| \left (    \int  (c_j-u)\1_{W_j(\cb)}(u)(\bar{X}_n - \mathbb{E}(X))(du)  \right )_{j=1, \hdots, k}  \right \|.
\]
We decompose $\|\hat{m}(\cb) - \cb^*\|^2$ as follows.
\begin{align}\label{eq:dec_ecart_one_step}
\|\hat{m}(\cb) - \cb^*\|^2 = \|\cb - \cb^*\|^2 + 2 \left\langle\hat{m}(\cb) - \cb,\cb-\cb^*\right\rangle + \|\hat{m}(\cb) - \cb\|^2.
\end{align}
Next, we bound the first term of \eqref{eq:dec_ecart_one_step}.
\begin{align*}
2 \left\langle\hat{m}(\cb) - \cb,\cb-\cb^*\right\rangle & = 2 \sum_{j=1}^{k} \frac{1}{\hat{p}_j(\cb)} \left\langle \int (u-c_j)\1_{W_j(\cb)}(u) \bar{X}_n(du), c_j - c_j^*\right\rangle \\
     & \begin{multlined}[t]\leq 2 \sum_{j=1}^{k} \frac{1}{\hat{p}_j(\cb)} \left\langle \int (u-c_j)\1_{W_j(\cb)}(u) \E(X)(du), c_j - c_j^*\right\rangle \\+ 2 D_n \sqrt{\sum_{j=1}^k \frac{\|c_j-c^*_j\|^2}{\hat{p}_j(\cb)^2}}
     \end{multlined} \\
     &\begin{multlined}[t]
     \leq 2 \sum_{j=1}^{k} \frac{1}{\hat{p}_j(\cb)} \left\langle p_j(\cb^*)(c_j^*-c_j), c_j - c_j^*\right\rangle \\
     + \frac{2 p_{min}}{8\sqrt{2}}\|\cb - \cb^*\|\sqrt{\sum_{j=1}^k \frac{\|c_j-c^*_j\|^2}{\hat{p}_j(\cb)^2}} + 2D_n \sqrt{\sum_{j=1}^k \frac{\|c_j-c^*_j\|^2}{\hat{p}_j(\cb)^2}}, 
     \end{multlined}
\end{align*}
where the last line follows from Lemma \ref{lem:gradient_margin}. Now, using Lemma \ref{lem:concentration_Lloyd} with $x= c_1 n p_{min}/M$, for $c_1$ a small enough absolute constant, entails that, with probability larger than $1-8e^{ c_1 n p_{min}/M}$, for $n$ large enough and every $\cb \in \B(\cb^*,R_0)$, 
\begin{align*}
\hat{p}_j(\cb) & \geq \frac{63}{64}p_j(\cb) - \frac{p_{min}}{64} \geq \frac{31}{32} p_{min} \\
\hat{p}_j(\cb) & \leq \frac{33}{32}p_j(\cb^*),
\end{align*} 
according to Lemma \ref{lem:gradient_margin}. Therefore
\begin{multline}\label{eq:Lloyd_oneiter_firstterm}
2 \left\langle\hat{m}(\cb) - \cb,\cb-\cb^*\right\rangle \leq -2 \sum_{j=1}^k \frac{p_j(\cb^*)}{\hat{p}_j(\cb)} \| c_j -c_j^*\|^2 + \frac{32}{124 \sqrt{2}} \| \cb - \cb^*\|^2 \\
 + K_1 \| \cb - \cb^*\|^2 + K_1^{-1}\frac{32^2}{31^2p_{min}^2}D_n^2,
\end{multline}
where $K_1 >0$ is to be fixed later. Then, the second term of \eqref{eq:dec_ecart_one_step} may be bounded as follows.
\begin{align*}
\|\hat{m}(\cb) - \cb\|^2 & = \sum_{j=1}^k \frac{\left \| \int (u-c_j)\1_{W_j(\cb)}(u)\bar{X}_n(du) \right \|^2}{\hat{p}_j(\cb)^2}  \\
           & = \sum_{j=1}^k \frac{\left \| p_j(\cb^*)(c_j-c_j^*) + \Delta_j(\cb) + \Delta_{n,j}(\cb) \right \|^2}{\hat{p}_j(\cb)^2},
\end{align*}
where
\begin{align*}
\Delta_j(\cb) =  \int \left [ (u-c_j)\1_{W_j(\cb)}(u)-(u-c_j^*)\1_{W_j(\cb^*)}(u) \right ]\E(X)(du),
\end{align*}
so that $\sum_{j=1}^{k} \|\Delta_j(\cb)\| \leq \frac{p_{min}}{8\sqrt{2}} \| \cb - \cb^*\|$, according to Lemma \ref{lem:gradient_margin}, and 
\begin{align*}
\Delta_{n,j}(\cb) = \int(u-c_j)\1_{W_j(\cb)}(u)(\bar{X}_n - \E(X))(du),
\end{align*}
so that $\sum_{j=1}^k \| \Delta_{n,j}\|^2 \leq D_n^2$. Thus, 
\begin{align}\label{eq:Lloyd_oneiter_secondterm}
\| \hat{m}(\cb) - \cb\|^2 & \leq 
\left ( 1 + K_2 + K_3 \right ) \sum_{j=1}^k \frac{p_j(\cb^*)^2}{\hat{p}_j(\cb)^2} \| c_j - c_j^*\|^2 + \left ( 1 + K_2^{-1} + K_4 \right ) \sum_{j=1}^{k} \frac{\|\Delta_j(\cb)\|^2}{\hat{p}_j(\cb)^2} \notag \\ 
& \quad  + \left ( 1 + K_3^{-1} + K_4^{-1} \right ) \sum_{j=1}^{k} \frac{\|\Delta_{n,j}(\cb)\|^2}{\hat{p}_j(\cb)^2}  \notag \\
                          &\leq \left ( 1 + K_2 + K_3 \right ) \sum_{j=1}^k \frac{p_j(\cb^*)^2}{\hat{p}_j(\cb)^2} \| c_j - c_j^*\|^2  + \left ( 1 + K_2^{-1} + K_4 \right ) \frac{32^2}{31^2 \times 128} \|\cb - \cb^*\|^2  \notag \\
                    & \quad  + \left ( 1 + K_3^{-1} + K_4^{-1} \right ) \frac{32^2}{31^2 p_{min}^2}D_n^2,
\end{align}
where $K_2$, $K_3$ and  $K_4$ are positive constants to be fixed later. Combining \eqref{eq:Lloyd_oneiter_firstterm} and \eqref{eq:Lloyd_oneiter_secondterm} yields that
\begin{align*}
\| \hat{m}(\cb) - \cb^*\|^2 & \leq \| \cb - \cb^*\|^2 \left ( 1 + K_1 + \frac{32}{124\sqrt{2}} + \frac{32^2}{31^2 \times 128} \left (1 + K_2^{-1} + K_4 \right ) \right ) \\
& \quad -2 \sum_{j=1}^{k} \frac{p_j(\cb^*)}{\hat{p}_j(\cb)} \| c_j-c_j^*\|^2 + \left ( 1 + K_2 + K_3 \right ) \sum_{j=1}^{k} \frac{p_j(\cb^*)^2}{\hat{p}_j(\cb)^2}\| c_j - c_j^*\|^2 \\
& \quad + D_n^2\frac{32^2}{31^2 p_{min}^2} \left ( 1+ K_1^{-1}   + K_3^{-1} + K_4^{-1} \right ). 
\end{align*}
Taking $K_2 = \frac{1}{32}$ gives, through numerical computation,
\begin{align*}
\| \hat{m}(\cb) - \cb^*\|^2 & \leq \| \cb-\cb^*\|^2 \left ( 0.62 + K_1 + K_3 \frac{32^2}{31^2} + K_4 \frac{32^2}{31^2 \times 128} \right ) \\
& \quad + D_n^2\frac{32^2}{31^2 p_{min}^2} \left ( 1+ K_1^{-1}   + K_3^{-1} + K_4^{-1} \right ) \\
& \leq \frac{3}{4}\| \cb - \cb^*\|^2 + \frac{K}{p_{min}^2}D_n^2,
\end{align*}
for $K_1$, $K_3$ and $K_4$ small enough.
Now, according to Lemma \ref{lem:concentration_Lloyd}, it holds
\[
\frac{K}{p_{min}^2}D_n^2 \leq \frac{R_0^2}{4},
\]
with probability larger than $1-e^{-{c_1 n p_{min}^2 \kappa_0^2/M^2 }}$, for some constant $c_1$ small enough.

Recalling that, for any $t$, $\cb^{(t+1)} = \hat{m}(\cb^{(t)})$, a straightforward recursion entails that, on a global probability event that has probability larger than $1-9e^{-cn \kappa_0^2 p_{min}^2/M^2}-e^{-x}$, for $c$ small enough, provided that $\cb^{(0)} \in \B(\cb^*,R_0)$, we have for any $t \geq 0$ $\cb^{(t)} \in \B(\cb^*,R_0)$ and 
\begin{align*}
\| \cb^{(t+1)} - \cb^*\| ^2 & \leq \frac{3}{4} \|\cb^{(t)} - \cb^*\|^2 + \frac{K}{p_{min}^2} D_n^2 \\
       & \leq \frac{3}{4} \|\cb^{(t)} - \cb^*\|^2 + \frac{K}{p_{min}^2} \frac{C^2 R^2 M^2}{n}  \left (k \sqrt{d\log(k)}+ \sqrt{x} \right )^2.
\end{align*}
\end{proof}
\subsubsection{Proof of Lemma \ref{lem:onestep_McQueen}}\label{tecsec:proof_lemma_onestep_McQueen}
\begin{lem*}[\ref{lem:onestep_McQueen}]
Assume that $\cb^{(0)} \in \mathcal{B}(\cb^*,R_0)$, and $n_t \geq c_0 \frac{k^2M^2}{p_{min}^2 \kappa_0^2} \log(n)$, for some constant $c_0$, with $n \geq k$. Then we have, for any $t =0, \hdots, T-1$,
\begin{align*}
\E \left ( \| \cb^{(t+1)} - \cb^*\|^2 \right ) \leq \left ( 1 - \frac{2-K_1}{t+1} \right ) \E \left ( \| \cb^{(t)} - \cb^* \|^2 \right ) + \frac{16kMR^2}{p_{min}(t+1)^2}, \qquad \mbox{\eqref{eq:mcqueen_iter_expectation}}
\end{align*}
with $K_1 \leq 0.5$.
\end{lem*}
The proof of Lemma \ref{lem:onestep_McQueen} will make use of the following deviation bounds.
\begin{lem}\label{lem:concentration_McQueen}
Let $\cb \in \B(0,R)^k$. Then, with probability larger than $1-2ke^{-x}$, we have, for all $j=1, \hdots, k$,
\begin{align*}
|\hat{p}_j(\cb) - p_j(\cb)| \leq \sqrt{\frac{2Mp_j(\cb)x}{n}} + \frac{Mx}{n}.
\end{align*}
Moreover, with probability larger than $1-e^{-x}$, we  have, 
\begin{align*}
\left \| \left ( \int (c_j-u)\1_{W_j(\cb)}(u)(\bar{X}_n-\E(X))(du) \right )_{j=1, \hdots, k} \right \| \leq \frac{4RM\sqrt{k}}{\sqrt{n}} \left ( 1 + \sqrt{x} \right ).
\end{align*}
\end{lem}   
A proof of Lemma \ref{lem:concentration_McQueen} is given in \Cref{sec:proof_lem_concentration_McQueen}.
\begin{proof}[Proof of Lemma \ref{lem:onestep_McQueen}]
Assume that $n \geq k$, and let $n_t =|B_t|/2 \geq \frac{C M }{p_{min}} \log(n)$, for $C$ large enough to be fixed later. For a given $t \leq T$, denote by $\hat{p}_j^{(t)}= \bar{X}_{B_t^{(1)}}(W_j(\cb^{(t)})$, and by $A_{t,1}$ and $A_{t,2}$ the events
\[
\begin{array}{@{}ccl} 
A_{t,1} &=& \left \{ \forall j = 1 \hdots, k \quad  |\hat{p}_j^{(t)} - p_j^{(t)} |  \leq  \frac{p_{min} + \sqrt{p_j^{(t)}p_{min}}}{256} \right \}, \\
A_{t,2} &=& \left \{ \forall j = 1 \hdots, k \quad  \left \|\int (c_j^{(t)}-u)\1_{W_j(\cb^{(t)})}(u)(\bar{X}_{B_t^{(2)}}-\E(X))(du)  \right \|  \leq 8 R \sqrt{\frac{Mkp_{min}}{C}} \right \}.
\end{array}
\] 
According to Lemma \ref{lem:concentration_McQueen} with $x= 4 \log(2n)$, we have $\mathbb{P}(A_{t,j}^c) \leq k/(2n^4)$, for $j=1,2$. We let also $A_{\leq t}$ denote the event $\cap_{r \leq t} A_{r_1} \cap A_{r_2}$, that has probability larger than $1- kt/n^4$.  First we prove that if $\cb^{(0)} \in \B(\cb^*,R_0)$, then, on $A_{\leq t}$, for all $r \leq t$, $\cb^{(r)} \in \B(\cb^*,R_0)$. We proceed recursively, assuming that $\cb^{(t)} \in \B(\cb^*,R_0)$. Then, on $A_{\leq t+1}$, applying Lemma \ref{lem:gradient_margin} yields that $\frac{33}{32}p_j^* \geq \hat{p}_j^t \geq \frac{31}{32}p_j^*$.  Denoting by $a_t = \| \cb^{(t)} - \cb^*\|^2$ and $g_{t+1}=\left (\frac{  \int (c^{(t)}_j-u)\1_{W_j(\cb^{(t)})}(u)\bar{X}_{B^{(2)}_{t+1}}(du)}{\hat{p}_j^{t+1}} \right )_j$, the recursion equation entails that
\begin{align}\label{eq:mcqueen_onestep}
a_{t+1} & = \|\pi_{B(0,R)^k} \left ( \cb^{(t)} - \frac{g_{t+1}}{t+1} \right )- \cb^* \|^2 \notag \\
& \leq \| \cb^{(t)} - \frac{g_{t+1}}{t+1} - \cb^*  \|^2 = a_t -  \frac{2}{t+1} \left\langle g_{t+1} , \cb-\cb^* \right\rangle + \frac{1}{(t+1)^2} \left \| g_{t+1} \right \|^2. 
\end{align}
As in the proof of Lemma \ref{lem:Lloyd_one_iteration}, denote by 
\begin{align*}
\Delta^t_j & =  \int (u-c_j^{(t)})\1_{W_j(\cb^{(t)})}(u)\E(X)(du)  - p_j^*(c_j^*-c_j^{(t)})\\
\Delta^{t+1}_{n,j} & = \int(u-c_j^{(t)})\1_{W_j(\cb^{(t)})}(u)(\bar{X}_{B^{(2)}_{t+1}}-\E(X))(du) \\
D_n^{t+1} &= \sqrt{\sum_{j=1}^k \| \Delta_{n,j}^{t+1}\|^2}.
\end{align*}
We have that
\begin{align*}
-  \frac{2}{t+1} \left\langle g_{t+1} , \cb-\cb^* \right\rangle & \begin{multlined}[t]
                                             \leq -\frac{2}{t+1} \sum_{j=1}^k \left ( \frac{p_j^*}{\hat{p}_j^{t+1}} \| c_j^{(t)} - c_j^*\|^2 - \frac{\| \Delta_{j,n}^{t+1} \| \| c_j^{(t)} - c_j^* \|}{\hat{p}_j^{t+1}} \right . \\
                                             - \left . \frac{\| c_j^{(t)}-c_j^*\|\| \Delta_j^{t}\|}{\hat{p}_j^{t+1}} \right )
\end{multlined} \\
  & \begin{multlined}[t] \leq -2 \frac{32}{33(t+1)} \| \cb^{(t)} - \cb^*\|^2 + \frac{64}{31p_{min}(t+1)} \| \cb^{(t)} - \cb^*\|D_n^{t+1} \\
   + \frac{64}{8\sqrt{2} \times 31 (t+1)} \| \cb^{(t)} -\cb^*\|^2
  \end{multlined} \\
   & \leq \|\cb^{(t)} - \cb^*\|^2 \left ( \frac{-64}{33(t+1)} + \frac{K_4}{t+1} + \frac{64}{8\sqrt{2} \times 31 (t+1)} \right ) +  K_4^{-1} \left ( \frac{32}{31 p_{min}} D_n^{t+1} \right )^2,
\end{align*}
according to Lemma \ref{lem:gradient_margin}, where $K_4$ denotes a constant. Next, the second term in \eqref{eq:mcqueen_onestep} may be bounded by
\begin{align*}
\left \| g_{t+1} \right \|^2 & \begin{multlined}[t] \leq  \sum_{j=1}^k \frac{1}{(\hat{p}_j^{t+1})^2} (p_j^*)^2 \|c_j^{(t)} - c_j^*\|^2 \left ( 1 + K_1 + K_2 \right ) \\
 + \frac{p_{min}^2}{128 \min_j{(\hat{p}_j^{t+1})^2}} \| \cb^{(t)} - \cb^*\|^2\left ( 1 + K_2^{-1} + K_3 \right ) +  \frac{1}{ \min_j{(\hat{p}_j^{t+1})^2}} \left ( 1 + K_1^{-1} + K_3^{-1} \right ) (D_n^{t+1})^2
\end{multlined} \\
 &\leq \frac{32^2}{31^2} \| \cb^{(t)} - \cb^* \|^2 \left ( 1 + K_1 + K_2 + \frac{1 + K_2^{-1} + K_3}{128} \right ) + (D_n^{t+1})^2 \frac{32^2(1 + K_1^{-1} + K_3^{-1})}{31^2 p_{min}^2}, 
\end{align*} 
where $K_1$, $K_2$ and $K_3$ are constants to be fixed later. Combining pieces and using $t+1 \geq 1$ leads to
\begin{multline*}
a_{t+1} \leq a_t + \frac{a_t}{t+1} \left( \frac{-64}{33} + \frac{64}{8\sqrt{2} \times 31} + K_4 +\frac{32^2}{31^2} \left ( 1 + K_1 + K_2 + \frac{1 + K_2^{-1} + K_3}{128} \right ) \right) \\ + \left(D_n^{t+1} \right )^2 \left (\frac{32^2}{31^2 p_{min}^2}K_4^{-1} + \frac{32^2(1 + K_1^{-1} + K_3^{-1})}{31^2 p_{min}^2} \right ).
\end{multline*}
Choosing $K_2= \frac{1}{32}$ entails that
\begin{multline*}
\left( \frac{-64}{33} + \frac{64}{8\sqrt{2} \times 31} + K_4 +\frac{32^2}{31^2} \left ( 1 + K_1 + K_2 + \frac{1 + K_2^{-1} + K_3}{128} \right ) \right) \\
 \leq -0.38 + K_4 + \frac{32^2}{31^2} \left ( K_1  + \frac{ K_3}{128} \right ),
\end{multline*}
so that, for $K_1$, $K_3$ and $K_4$ small enough, we have
\begin{align*}
a_{t+1} \leq 0.8 a_t + \frac{K}{p_{min}^2} (D_n^{t+1})^2.
\end{align*} 
Now, if $n_t \geq c_0 \frac{k^2M^2}{p_{min}^2 \kappa_0^2} \log(n)$, $n \geq k$, where $c_0$ is  an absolute constant, on $A_{\leq t+1}$ we have
\begin{align*}
a_{t+1} \leq 0.8 a_t + 0.2 R_0^2 \leq R_0^2. 
\end{align*}
Thus, provided that $\cb^{(0)} \in \mathcal{B}(\cb^*,R_0)$, on $A_{\leq t}$ we have $\cb^{(p)} \in \mathcal{B}(\cb^*,R_0)$, for $ p \leq t$.

Next, if $\mathcal{F}_{t}$ denotes the sigma-algebra corresponding to the observations of the $t$ first mini-batches $B_1, \hdots, B_t$, and $E_t$ denotes the conditional expectation with respect to $\mathcal{F}_{t}$. We will show that 
\begin{align*}
\E a_{t+1} \leq \left ( 1 - \frac{2 - K_1}{t+1} \right ) \E a_t + \frac{16kMR^2}{p_{min}(t+1)^2},
\end{align*} 
where $K_1 < 0.5$. First, we may write 
\begin{align*}
\E a_{t+1} & = \E \left ( a_{t+1} \1_{A_{\leq t}} \1_{A_{t+1,1}} \right ) + R_1, 
\end{align*}
with $R_1 \leq (4k^2R^2)/n^3 \leq 4kMR^2/(p_{min} (t+1)^2)$, since $p_{min} \leq M/k$ and $t+1 \leq T \leq n$.  Then, using \eqref{eq:mcqueen_onestep} entails
\begin{multline*}
\E \left ( a_{t+1} \1_{A_{\leq t}\cap A_{t+1,1}} \right ) \leq  \E a_t - \frac{2}{t+1}\E \left ( \left\langle g_{t+1},\cb^{(t)} - \cb^* \right\rangle  \1_{A_{\leq t}\cap A_{t+1,1}} \right ) \\ 
+ \frac{1}{(t+1)^2} \E \left ( \|g_{t+1}\|^2  \1_{A_{\leq t}\cap A_{t+1,1}} \right ).
\end{multline*}
Next, we bound the scalar product as follows.
\begin{multline*}
\E \left ( \left\langle - g_{t+1},\cb^{(t)} - \cb^* \right\rangle \1_{A_{\leq t}\cap A_{t+1,1}} \right ) \\ = \E \left ( \1_{A_{\leq t}}\sum_{j=1}^k   E_t \left ( \frac{\left\langle   \int (u - c^{(t)}_j)\1_{W_j(\cb^{(t)})}(u)\bar{X}_{B^{(2)}_{t+1}}(du), c_j^{(t)} - c_j^* \right\rangle}{\hat{p}_j^{t+1}} \1_{A_{t+1,1}} \right ) \right ),
\end{multline*}
where, for any $j \in [\![1,k]\!]$, we have
\begin{multline*}
E_t \left ( \frac{\left\langle   \int (u - c^{(t)}_j)\1_{W_j(\cb^{(t)})}(u)\bar{X}_{B^{(2)}_{t+1}}(du), c_j^{(t)} - c_j^* \right\rangle}{\hat{p}_j^{t+1}} \1_{A_{t+1,1}} \right ) \\
= E_t \left (\left\langle   \int (u - c^{(t)}_j)\1_{W_j(\cb^{(t)})}(u)\bar{X}_{B^{(2)}_{t+1}}(du), c_j^{(t)} - c_j^* \right\rangle \right ) E_t \left (\frac{1}{\hat{p}_j^{t+1}} \1_{A_{t+1,1}} \right ).
\end{multline*}
The first term may be bounded by
\begin{multline*}
E_t \left (\left\langle   \int (u - c^{(t)}_j)\1_{W_j(\cb^{(t)})}(u)\bar{X}_{B^{(2)}_{t+1}}(du), c_j^{(t)} - c_j^* \right\rangle \right ) \1_{A_{\leq t}} \leq - p_j^* \| c_j^{(t)} - c_j^*\|^2 \1_{A_{\leq t}} \\ + \|\Delta_{j,t}\| \|c_j^{(t)} - c_j^*\|,
\end{multline*}
and the second term by
\begin{align*}
\frac{32}{33p_j^*}\1_{A_{\leq t}} \leq E_t \left (\frac{1}{\hat{p}_j^{t+1}} \1_{A_{t+1,1}} \right ) \1_{A_{\leq t}} \leq \frac{32}{31p_j^*}.
\end{align*}
This gives
\begin{align*}
\E \left ( \left\langle - g_{t+1},\cb^{(t)} - \cb^* \right\rangle \1_{A_{\leq t}\cap A_{t+1,1}} \right ) & \leq - \frac{32}{33} \E \left (a_t \1_{A_{\leq t}} \right ) + \frac{32}{31 p_{min}} \E \left (  \sum_{j=1}^k \| \Delta_{j,t} \| \|c_j^{(t)} - c_j^*\| \1_{A_{\leq t}} \right ) \\
& \leq - \frac{32}{33} \E(a_t) + \frac{32}{31 \times 8 \sqrt{2}} \E(a_t) + \frac{32}{33} \frac{4k^2R^2}{n^3},
\end{align*}
according to Lemma \ref{lem:gradient_margin}. At last, the bound on $\|g_{t+1}\|^2$ writes as follows.
\begin{align*}
\E \left ( \|g_{t+1}\|^2 \1_{A_{\leq t} \cap A_{t+1,1}} \right ) & = \E \left ( \1_{A_{\leq t}} E_t \left ( \|g_{t+1}\|^2 \1_{A_{t+1,1}}\right ) \right ) \\
&\leq \frac{32}{31} \frac{4kMR^2}{p_{min}}.
\end{align*}
Gathering all pieces leads to
\begin{align*}
\E a_{t+1} \leq \left ( 1 - \frac{2-K_1}{t+1} \right ) \E a_t + \frac{16 kMR^2}{p_{min}(t+1)^2},
\end{align*}
with $K_1 \leq 0.5$.

At last, in the point sample case where we observe $n$ points in $\mathbb{R}^d$ $X_1, \hdots, X_n$ i.i.d with distribution $X$, recall that we take
\[
\cb^{(t+1)} = \cb^{(t)} - \frac{g_{t+1}}{t+1}, \quad 
\mbox{}with \quad  
g_{t+1} = \frac{ \int (c_j^{(t)}-u)\mathds{1}_{W_j(\cb^{(t)}}(u)\bar{X}_{B_{t+1}}(du)}{\hat{p}_j^{t+1}}.
\] 
With a slight abuse of notation we denote by $A_{\leq t}$ the event onto which the concentration inequalities of Lemma \ref{lem:concentration_McQueen} are satisfied (with $B_{t}^{(1)} = B_{t}^{(2)} = B_{t}$). The first step of the proof is the same: if $n_{t} \geq c_0 \frac{k}{p_{min}^2 \kappa_0^2}$ and $ n \geq k$, then, on $A_{\leq t}$, for all $j \leq t$, $a_j \leq R_0$. It remains to prove the recursion inequality
\[
\mathbb{E}a_{t+1} \leq \left ( 1 - \frac{2-K_1}{t+1} \right ) \mathbb{E}a_t + \frac{16kR^2}{p_{min}(t+1)^2}.
\]
We proceed as before, writing $\E(a_{t+1}) = \E(a_{t+1}\mathds{1}_{A_{\leq t}}) + R_1$, with $R_1 \leq 4kR^2/(p_{min}(t+1)^2)$, and
\begin{align*}
\E(a_{t+1}\mathds{1}_{A_{\leq t}}) \leq \E a_t + \frac{1}{(t+1)^2} \E (\|g_{t+1}\|^2)- \frac{2}{t+1} \E \left ( \left\langle g_{t+1}, \cb^{(t)} - \cb^* \right\rangle \1_{A_{\leq t}} \right ). 
\end{align*} 
Note that $\|g_{t+1}\|^2 \leq 4 k R^2$, so that the second term might be bounded by $4kR^2/(t+1)^2$. To bound the scalar product, we proceed as follows. Let $j \in [\![1,k]\!]$, and, for $i \in B_{t+1}$, denote by $U_{i,j}$ the random variable $\1_{W_j(\cb^{(t)})}(X_i)$. We have
\begin{align*}
\E & \left ( \frac{\left\langle  \int (u - c_j^{(t)}) \1_{W_j(\cb^{(t)}}(u)\bar{X}_{B_{t+1}}(du), c_j^{(t)}- c_j^* \right\rangle} {\hat{p}_j^{t+1}} \1_{A_{\leq t}} \right ) \\
&  = \E \left ( \frac{\1_{A _{\leq t}}}{\hat{p}_{j}^{t+1}} \left\langle \mathbb{E} \left [  \int (u - c_j^{(t)}) \1_{W_j(\cb^{(t)})}(u)\bar{X}_{B_{t+1}}(du) \mid \mathcal{F}_t,(U_{i,j})_{i \in B_{t+1}} \right ] , c_j^{(t)}- c_j^* \right\rangle \right ) \\
 & = \E \left ( \frac{\1_{A _{\leq t}}}{n_{t+1} \hat{p}_{j}^{t+1}} \sum_{i \in B_{t+1}} U_{i,j} \left[ - \frac{p_j^*}{p_j^t} \| c_j^{(t)} - c_j^*\|^2 + \frac{\left\langle \Delta_{j,t}, c_j^{(t)} - c_j^* \right\rangle}{p_j^t} \right ]  \right ) \\
& \leq -\frac{32}{33}\E ( \|c_j^* - c_j\|^2 \1_{A_{\leq t}} ) + \frac{32}{31 p_{min}} \E ( \| \Delta_{j,t}\| \|c_j^t - c_j^* \| \1_{A_{\leq t}} ). 
\end{align*}
The remaining of the proof is the same as for the sample measure case.
\end{proof}
\subsubsection{Proof of Lemma \ref{lem:concentration_Lloyd}}\label{sec:proof_lem_concentration_Lloyd}
Let $Z_1$ denote the process
\[
Z_1 = \sup_{\cb \in \B(0,R)^k, j=1, \hdots, k} \left |  \int \1_{W_j(\cb)}(u)\left ( \frac{\bar{X}_n}{M} - \frac{\mathbb{E}(X)}{M} \right )(du) \right |.
\]
Note that the VC dimension of Voronoi cells in a $k$-points Voronoi diagram is at most $c_0 kd \log(k)$ (\cite[Theorem 1.1]{vanderVaart09}). We first use a symmetrization bound.
\begin{lem}\label{lem:symmetrizationnromalized}
Let $\mathcal{F}$ denote a class of functions taking values in $[0,1]$, and $X_1, \hdots, X_n$, $X'_1, \hdots, X'_n$ i.i.d random variables drawn from $P$. Denote by $P_n$ and $P'_n$ the empirical distributions associated to the $X_i$'s and $X'_i$'s. If $nt^2 \geq 1$, then
\begin{align*}
\mathbb{P} \left ( \sup_{f \in \mathcal{F}} \frac{(P-P_n)f}{\sqrt{Pf}} \geq 2t \right ) & \leq 2 \mathbb{P} \left ( \sup_{f \in \mathcal{F} } \frac{(P'_n-P_n)f}{\sqrt{(P'_n f + P_nf)/2}} \geq t \right ) \\
\mathbb{P} \left ( \sup_{f \in \mathcal{F}} \frac{(P_n-P)f}{\sqrt{P_nf}} \geq 2t \right ) & \leq 2 \mathbb{P} \left ( \sup_{f \in \mathcal{F} } \frac{(P_n-P'_n)f}{\sqrt{(P'_n f + P_nf)/2}} \geq t \right ).
\end{align*}
\end{lem}
For the sake of completeness a proof of Lemma \ref{lem:symmetrizationnromalized} is given in \Cref{sec:proof_lem_symmetrization_normalized}. Next, introducing $\sigma_1, \hdots,  \sigma_n$ independent Rademacher variables, we get
\begin{align*}
\mathbb{P} \left ( \sup_{f \in \mathcal{F} } \frac{(P'_n-P_n)f}{\sqrt{(P'_n f + P_nf)/2}} \geq t \right ) & \leq \mathbb{P} \left ( \sup_{f \in \mathcal{F} } \frac{\frac{1}{n} \sum_{i=1}^n{\sigma_i (f(X_i) - f(X'_i))}}{\sqrt{(P'_n f + P_nf)/2}} \geq t \right ) \\
& \leq \mathbb{E}_{X_1, \hdots, X_n, X'_1, \hdots, X'_n} \left ( \mathbb{P}_{\sigma} \left ( \sup_{f \in \mathcal{F} } \frac{\frac{1}{n} \sum_{i=1}^n{\sigma_i (f(X_i) - f(X'_i))}}{\sqrt{(P'_n f + P_nf)/2}} \geq t \right ) \right ). 
\end{align*}
For a set of functions $\mathcal{F}$ and elements $x_1, \hdots, x_q \in \mathcal{M}(R)$ we denote by $S_{\mathcal{F}}(x_1, \hdots, x_q)$ the cardinality of the set $\{ (f(x_1), \hdots, f(x_q))\mid f \in \mathcal{F} \}$.  Let $\mathcal{F}_1$ denote the set of functions $\{X \in \mathcal{M}(R,M) \mapsto X(W)/M \mid W= \bigcap_{j=1}^k H_j, \mbox{ $H_j$ half-space} \}$. Since, for every $i \in [\![1,n]\!]$, $X_i = \sum_{j=1}^{n_i} \mu_{i,j} \delta_{x^{(i)}_j}$, we have 
\begin{align*}
S_{\mathcal{F}_1}(X_1, \hdots, X_n, X'_1, \hdots, X'_n) & \leq |\{(\mathds{1}_W(x^{(i)}_j))_{i=1, \hdots, 2n, j=1, \hdots, n_i} \mid W= \bigcap_{j=1}^k H_j, \mbox{ $H_j$ half-space} \}| \\ 
& \leq \left ( 2 \left ( \sum_{i=1}^n n_i + n'_i \right ) \right )^{c_0 kd \log(k)},
\end{align*}
using \cite[Theorem 1]{Mendelson03}, and \cite[Theorem 1]{vanderVaart09} to bound the VC-dimension of the sets $W$'s. On the other hand, for any $f \in \mathcal{F}_1$, it holds
\begin{align*}
\frac{\sum_{i=1}^n (f(X'_i) - f(X_i))^2}{\sum_{i=1}^n (f(X_i) + f(X'_i))} \leq 1. 
\end{align*}
Thus, combining Hoeffding's inequality and a plain union bound yields
\[
\left ( \mathbb{P}_{\sigma} \left ( \sup_{f \in \mathcal{F} } \frac{\frac{1}{n} \sum_{i=1}^n{\sigma_i (f(X_i) - f(X'_i))}}{\sqrt{(P'_n f + P_nf)/2}} \geq t \right ) \right ) \leq \left ( 2 \sum_{i=1}^n(n_i + n'_i) \right )^{c_0 kd \log(k)} e^{-nt^2},
\]
hence, since for $i=1, \hdots, n$, $X_i \in \mathcal{M}_{N_{max}}(R,M)$,
\[
\mathbb{P} \left ( \sup_{f \in \mathcal{F} } \frac{(P'_n-P_n)f}{\sqrt{(P'_n f + P_nf)/2}} \geq t \right ) \leq (4 n N_{max})^{c_0 kd \log(k)} e^{-nt^2},
\]
that proves the second inequality of Lemma \ref{lem:concentration_Lloyd}. The first inequality of Lemma \ref{lem:concentration_Lloyd} derives the same way from the second inequality of Lemma \ref{lem:symmetrizationnromalized}.

We turn to the third inequality of Lemma \ref{lem:concentration_Lloyd}. Let $Z$ denote the process
\[
Z = \sup_{\cb \in \B(0,R)^k, \|\mathbf{t}\| \leq 1} \left\langle \left ( \int   (c_j-u)\1_{W_j(\cb)}(u) \left (\frac{\bar{X}_n}{M} - \frac{\mathbb{E}(X)}{M} \right )(du) \right )_{j=1, \hdots, k},\mathbf{t} \right\rangle, 
\]
and, for $j=1, \hdots, k$, 
\[
Z_j = \sup_{\cb \in \B(0,R)^k, \|t_j\| \leq 1} \left\langle \frac{1}{M} \int (c_j-u)\1_{W_j(\cb)}(u)(\bar{X}_n - \mathbb{E}(X))(du) ,t_j \right\rangle,
\]
so that 
$Z \leq \sqrt{\sum_{j=1}^{k} Z_j^2}$. 
According to the bounded differences inequality (\cite[Theorem 6.2]{Massart13}), we have
\begin{align*}
\mathbb{P} \left ( Z_j \geq \mathbb{E}(Z_j) + \sqrt{\frac{8 R^2}{n}x} \right ) \leq e^{-x}.
\end{align*}
Using symmetrization we get 
\begin{align*}
\mathbb{E}Z_j \leq \frac{2}{n} \mathbb{E}_{X_1, \hdots, X_n} \mathbb{E}_\sigma \sup_{\cb \in \B(0,R)^k, \|t\| \leq 1} \sum_{i=1}^n { \sigma_i \left\langle \frac{1}{M} \int (c_j-u)\1_{W_j(\cb)}(u)X_i(du) ,t_j \right\rangle}, 
\end{align*}
where $\sigma_1, \hdots, \sigma_n$ are i.i.d Rademacher variables. Now assume that $X_1, \hdots, X_n$ is fixed and $j=1$. For a set $\mathcal{F}$ of real-valued functions we denote by $\mathcal{N}(\mathcal{F}, \varepsilon, \|.\|)$ its $\varepsilon$-covering number with respect to the norm $\|.\|$. Denoting by $\Gamma_0$, $\Gamma_1$ and $\Gamma_2$ the following sets
\begin{align*}
\begin{array}{@{}ccc}
\Gamma_0 &=&  \left \{ \gamma^{(0)}_{(\cb,t_1)}:X \mapsto \frac{1}{M} \int \frac{\left\langle  c_1-u,t_1  \right\rangle}{2R} \1_{W_1(\cb)}(u) X(du)  \mid \cb \in \B(0,R)^k, t_1 \in \B(0,1) \right \}, \\
\Gamma_1 & =& \left \{ \gamma^{(1)}_{(c_1,t_1)}: x \mapsto \frac{\left\langle c_1-x,t_1 \right\rangle}{2R} \mid c_1 \in \B(0,R), t_1 \in \B(0,1) \right \}, \\
\Gamma_2 & =& \left \{ \gamma^{(2)}_{\cb'}:x \mapsto \1_{W_1(\cb)}(x) \mid \cb \in \B(0,R)^k \right \},
\end{array}
\end{align*}
so that, for every $(\cb,t_1), (\cb',t'_1) \in (\mathcal{B}(0,R)^k \times \mathcal{B}(0,1))^2$ , 
\[
\gamma^{(0)}_{(\cb,t_1)}(X_i) - \gamma^{(0)}_{(\cb',t'_1)}(X_i) = \frac{1}{M} \int \left [ \gamma^{(1)}_{(c_1,t_1)}(u)\gamma^{(2)}_{\cb}(u) - \gamma^{(1)}_{(c'_1,t'_1)}(u)\gamma^{(2)}_{\cb'}(u) \right ]X_i(du).
\]
Let $\varepsilon>0$. If $ \|\gamma^{(1)}_{(c_1,t_1)} - \gamma^{(1)}_{(c'_1,t'_1)}\|_\infty \leq \varepsilon$, we may write
\begin{align*}
(\gamma^{(0)}_{(\cb,t_1)}(X_i) - \gamma^{(0)}_{(\cb',t'_1)}(X_i))^2 \leq \left ( \varepsilon + \frac{1}{M} \int  |\gamma^{(2)}_{\cb}(u)-\gamma^{(2)}_{\cb'}(u)|X_i(du)  \right )^2. 
\end{align*}
Thus, 
\begin{align*}
\| \gamma^{(0)}_{(\cb,t_1)} - \gamma^{(0)}_{(\cb',t'_1)} \|^2_{L_2(P_n)} & \leq 2 \varepsilon^2 + \frac{2}{n} \sum_{j=1}^n  \|\gamma^{(2)}_{\cb}-\gamma^{(2)}_{\cb'}\|^2_{L_2(X_i/M)} \\
& \leq 2 \varepsilon^2 + 2 \|\gamma^{(2)}_{\cb}-\gamma^{(2)}_{\cb'}\|^2_{L_2(\bar{X}_n/M)} \\
& \leq 2 \varepsilon^2 + 2 \|\gamma^{(2)}_{\cb}-\gamma^{(2)}_{\cb'}\|^2_{L_2(\bar{X}_n/M(\bar{X}_n))} . 
\end{align*}
We deduce
\[
\mathcal{N}(\Gamma_0,\varepsilon, L_2(P_n)) \leq \mathcal{N}(\Gamma_1, \varepsilon /2, \|.\|_\infty) \times \mathcal{N}(\Gamma_2, \varepsilon /2, L_2(\bar{X}_n/M(\bar{X}_n))),
\]
for every $\varepsilon >0$.  
According to \cite[Theorem 1]{Mendelson03}, we may write
\begin{align*}
\mathcal{N} \left ( \Gamma_1 , \frac{\varepsilon}{2}, \|.\|_\infty \right ) & \leq \left ( \frac{4}{\varepsilon} \right )^{K(d+1)}, \\
\mathcal{N} \left ( \Gamma_2 , \frac{\varepsilon}{2}, L_2(\bar{X}_n/M(\bar{X}_n)) \right ) & \leq \left ( \frac{4}{\varepsilon} \right )^{c_0 K kd \log(k)},
\end{align*}
where $K$ is a constant and $\varepsilon < 2$. Thus, for every $\varepsilon < 2$, 
\begin{align*}
\mathcal{N}\left ( \Gamma_0, \varepsilon, L_2(P_n) \right ) \leq \left ( \frac{4}{\varepsilon} \right )^{C kd \log(k)}.
\end{align*}
Using Dudley's entropy integral (see, e.g., \cite[Corollary 13.2]{Massart13}) yields, for $k\geq 2$,
\begin{align*}
\mathbb{E}_\sigma Z_j \leq C R \sqrt{\frac{kd \log(k)}{{n}}},
\end{align*}
hence the result.

\subsubsection{Proof of Lemma \ref{lem:concentration_McQueen}}\label{sec:proof_lem_concentration_McQueen}
The first bound of Lemma \ref{lem:concentration_McQueen} follows from Bernstein's inequality. 
To prove the second inequality, we first bound the expectation as follows.
\begin{multline*}
\mathbb{E} \left ( \left \|  \left ( \int (c_j-u)\1_{W_j(\cb)}(u)(\bar{X}_n-\E(X))(du) \right )_{j=1, \hdots, k} \right \| \right ) \\
\begin{aligned}[t]
&\leq \sqrt{\E \left \|  \left ( \int (c_j-u)\1_{W_j(\cb)}(u) (\bar{X}_n-\E(X))(du) \right )_{j=1, \hdots, k} \right \|^2 } \\
 & \leq \sqrt{\frac{1}{n^2} \sum_{i=1}^n \E \left ( \left \|  \left ( \int (c_j-u)\1_{W_j(\cb)}(u) (X_i-\E(X))(du) \right )_{j=1, \hdots, k} \right \|^2 \right )} \\
& \leq \sqrt{\frac{(4RM)^2k}{n}} = \frac{4RM\sqrt{k}}{\sqrt{n}}
\end{aligned}
\end{multline*}
A bounded difference inequality (see, e.g., \cite[Theorem 6.2]{Massart13} entails that, with probability larger than $1-e^{-x}$, 
\begin{align*}
\left \|  \left ( \int (c_j-u)\1_{W_j(\cb)}(u)(\bar{X}_n-\E(X))(du) \right )_{j=1, \hdots, k} \right \| \leq  \frac{4RM\sqrt{k}}{\sqrt{n}} + \sqrt{\frac{8kR^2M^2x}{n}},
\end{align*}
hence the result.
\subsubsection{Proof of Lemma \ref{lem:symmetrizationnromalized}}\label{sec:proof_lem_symmetrization_normalized}
We follow the proof of \cite[Theorem 1]{Bartlett99}. Let $t >0$, and assume that $Pf - P_nf > 2t \sqrt{Pf}$, as well as $P'_nf \geq Pf - t \sqrt{Pf} \geq 0$. Let $g_a: \mathbb{R}^+ \rightarrow \mathbb{R}$ be defined as $g_a(x) = \frac{x-a}{\sqrt{x+a}}$, for $a \geq 0$. Then $g_a$ is nondecreasing on $\mathbb{R}^+$. With $a = P_nf$, $0 \leq x_2 = Pf - t \sqrt{Pf} \leq P'_nf = x_1$, we have $g_a(x_2) \leq g_a(x_1)$, so that
\begin{align*}
\frac{P'_nf - P_nf}{\sqrt{\frac{1}{2}(P'_nf + P_nf)}} \geq \frac{Pf - t \sqrt{Pf} - P_nf}{\sqrt{\frac{1}{2}(P_nf + Pf - t \sqrt{Pf}}}.
\end{align*} 
Since $P_nf + Pf - t \sqrt{Pf} \leq 2 Pf$, we deduce that
\begin{align*}
\frac{P'_nf - P_nf}{\sqrt{\frac{1}{2} \left ( P'_nf + P_nf \right )}} \geq \frac{P_f - P_nf - t \sqrt{Pf}}{\sqrt{Pf}} \geq t.
\end{align*}
Thus, 
\begin{multline*}
\mathbb{P} \left ( \sup_{f \in \mathcal{F}} \frac{P'_nf - P_nf}{\sqrt{\frac{1}{2} \left ( P'_nf + P_nf \right )}} \right ) \\
\geq \mathbb{P} \left ( \exists f_0 \in \mathcal{F} \mid Pf_0 - P_n f_0 > 2t \sqrt{Pf_0} \mbox{ and } P'_nf_0 \geq Pf_0 - t \sqrt{Pf_0} \right ).
\end{multline*}
Let $A = \left \{ \exists f_0 \in \mathcal{F} \mid Pf_0 - P_n f_0 > 2t \sqrt{Pf_0} \right \}$. The previous inequality may be written as
\begin{multline*}
\mathbb{P} \left ( \exists f_0 \in \mathcal{F} \mid Pf_0 - P_n f_0 > 2t \sqrt{Pf_0} \mbox{ and } P'_nf_0 \geq Pf_0 - t \sqrt{Pf_0} \right ) \\
\geq \mathbb{E} \left ( \mathds{1}_A \mathbb{P} \left ( \bigcup_{\{f \mid Pf - P_n f > 2t \sqrt{Pf} \}} \{P'_nf \geq Pf - t \sqrt{Pf}\} \,\middle|\, X_1, \hdots, X_n \right ) \right ),
\end{multline*}
with, for a fixed $f \in \mathcal{F}$, using Cantelli's inequality,
\begin{align*}
\mathbb{P} \left ( P'_nf - Pf < -2t \sqrt{Pf} \right ) \leq \frac{\frac{1}{n} \Var(f)}{\frac{1}{n}\Var(f) + t^2 Pf}.
\end{align*}
Since, for any $f \in \mathcal{F}$, $\|f\|_\infty \leq 1$, we have $\Var(f) \leq Pf^2 \leq Pf$, thus
\begin{align*}
\mathbb{P} \left ( P'_nf - Pf < -t \sqrt{Pf} \right ) \leq \frac{1}{2},
\end{align*}
provided that $nt^2 \geq 1$. Thus
\begin{align*}
\mathbb{P} \left ( \sup_{f \in \mathcal{F}} \frac{P'_nf - P_nf }{\sqrt{\frac{1}{2}( P'_nf + P_nf)}} \geq t \right ) \geq \frac{1}{2} \mathbb{P} \left ( \frac{Pf - P_nf}{\sqrt{Pf}} \geq 2t \right ).
\end{align*}
The other inequality proceeds from the same reasoning, considering $f$ such that $P_nf - Pf > 2t \sqrt{P_nf}$ and $P'_nf \leq Pf + t \sqrt{P_nf}$, and $g_a: \mathbb{R}^+ \rightarrow \mathbb{R}$ defined by $g_a(x) = \frac{a-x}{\sqrt{a+x}}$, that is nonincreasing for $a \geq 0$. Choosing $a = P_nf$, $0 \leq x_1 = P'_nf \leq Pf + t \sqrt{
P_nf} = x_2$ leads to
\begin{align*}
\frac{P_nf - P'_nf}{\sqrt{\frac{1}{2} \left ( P'_nf + P_nf \right )}} \geq \frac{P_nf - Pf - t \sqrt{P_nf}}{\sqrt{\frac{1}{2} \left ( Pf + t \sqrt{P_nf} + P_nf\right ) } } \geq \frac{P_nf - Pf - t \sqrt{P_nf}}{\sqrt{P_nf}} \geq t,
\end{align*}
since $Pf + t \sqrt{P_nf} \leq P_nf$. Using Cantelli's inequality again leads to the result.

\subsection{Proof of Lemma \ref{lemma-truncated-dgm}}\label{sec:proof_lemma_truncated_dgm}
\begin{lem*}[\ref{lemma-truncated-dgm}]
Let $S$ be a compact subset of $\R^d$, and $D$ denote the persistence diagram of the distance function $d_S$. For any $s >0$, the truncated diagram consisting of the points $m=(m^1, m^2) \in D$ such that $m^2 - m^1 \geq s$ is finite.    
\end{lem*}
The lemma follows from standard arguments in geometric inference and persistent homology theory. 
  
First, the definition of generalized gradient of $d_S$ - see \cite{chazal2009sampling} or \cite[Section 9.2]{boissonnat2018geometric}  - implies that the critical points of $d_S$ are all contained in the convex hull of $S$. As a consequence, they are all contained in the sublevel set $d_S^{-1}([0, 2{\rm diam}(S)])$. It follows from the Isotopy Lemma - \cite[Theorem 9.5]{boissonnat2018geometric}  - that all the sublevel sets $d_S^{-1}([0, t])$, $t > 2{\rm diam}(S)$ have the same homology.  As a consequence, no point in $D$ has a larger coordinate than $2{\rm diam}(S)$ and $D$ is contained in $[0,2{\rm diam}(S)]^2$. 
             
Since $S$ is compact, the persistence module of the filtration defined by the sublevel sets of $d_S$ is $q$-tame (\cite[Corollary 3.35]{Chazal2016}). Equivalently, this means that for any $b_0 < d_0$, the intersection of $D$ with the quadrant $Q_{(b_0,d_0)} = \{ (b,d): b<b_0 \ {\rm and} \  d_0 < d \}$ is finite. 
Noting that the intersection of $[0,2{\rm diam}(S)]^2$ with the half-plane $\{ (b,d) : d \geq b + s \}$ can be covered by a finite union of quadrants $Q_{(b,b+\frac{s}{2})}$ concludes the proof of the lemma.

\end{document}